\documentclass{amsart}
\usepackage{a4}
\usepackage{graphics}
\usepackage{color}
\usepackage{amssymb}
\usepackage[all]{xy}
\usepackage{amsmath}
\usepackage{stmaryrd}
\usepackage{longtable}
\usepackage{booktabs}

\usepackage[backrefs,lite]{amsrefs}

\usepackage{tikz}

\usetikzlibrary{patterns,snakes}
\usetikzlibrary{calc} 

\CompileMatrices

\newtheorem{theorem}{Theorem}[section]
\newtheorem{lemma}[theorem]{Lemma}
\newtheorem{proposition}[theorem]{Proposition}
\newtheorem{corollary}[theorem]{Corollary}

\theoremstyle{definition}

\newtheorem{example}[theorem]{Example}

\newtheorem{construction}[theorem]{Construction}

\newtheorem{remark}[theorem]{Remark}
\theoremstyle{remark}

\numberwithin{equation}{section}
\def\quot{/\!\!/}

\def\bangle#1{\langle #1 \rangle}
\def\sei{\mathrel{\mathop:}=}

\def\KK{{\mathbb K}}

\def\ZZ{{\mathbb Z}}

\def\QQ{{\mathbb Q}}
\def\PP{{\mathbb P}}

\def\Eff{{\rm Eff}}
\def\Mov{{\rm Mov}}

\def\Ample{{\rm Ample}}
\def\SAmple{{\rm SAmple}}
\def\BPF{{\rm BPF}}

\def\cov{{\rm cov}}

\def\Cl{\operatorname{Cl}}

\def\Pic{\operatorname{Pic}}

\def\Spec{{\rm Spec}}
\def\Proj{{\rm Proj}}

\def\conv{{\rm conv}}
\def\cone{{\rm cone}}
\def\lin{{\rm lin}}

\def\pr{{\rm pr}}

\def\cov{{\rm cov}}

\def\relint{{\rm relint}}

\def\tm{{\tau^-}}
\def\tp{{\tau^+}}

\allowdisplaybreaks


\begin{document}
\title[On intrinsic quadrics]%
{On intrinsic quadrics}
\author[A.~Fahrner]{Anne Fahrner}
\address{Mathematisches Institut, Universit\"at T\"ubingen,
Auf der Morgenstelle 10, 72076 T\"ubingen, Germany}
\thanks{Supported by the Carl-Zeiss-Stiftung}
\email{fahrner@math.uni-tuebingen.de}
\author[J.~Hausen]{J\"urgen Hausen}
\address{Mathematisches Institut, Universit\"at T\"ubingen,
Auf der Morgenstelle 10, 72076 T\"ubingen, Germany}
\email{juergen.hausen@uni-tuebingen.de}

\begin{abstract}
An intrinsic quadric is a normal projective
variety with a Cox ring defined by a single 
quadratic relation.
We provide explicit descriptions of these 
varieties in the smooth case for 
small Picard numbers.
As applications we figure out 
in this setting the Fano examples 
and (affirmatively) test Fujita's 
freeness conjecture.
\end{abstract}

\maketitle

\section{Introduction}

Intrinsic quadrics have been introduced 
in~\cite{BeHa:2007} 
as an example class of normal, projective algebraic 
varieties that are accessible by elementary 
combinatorial methods similar to toric varieties.
Recall that the normal projective toric 
varieties~$X$ are characterized by the property 
that their divisor class group $\Cl(X)$ is finitely
generated and their Cox~ring 
$$ 
\mathcal{R}(X)  
\ = \ 
\bigoplus_{\Cl(X)} \Gamma(X,\mathcal{O}_X(D))
$$
is a polynomial ring. An \emph{intrinsic quadric} 
is by definition a normal projective variety~$X$ 
with finitely generated divisor class 
group $\Cl(X)$ and a finitely generated Cox ring 
admitting homogeneous generators such that the 
associated ideal of relations is generated 
by a single, purely quadratic polynomial.  
In that sense, studying intrinsic quadrics
is a quite controled step beyond toric geometry.
Some well known non-toric examples are 
the usual smooth quadrics $X \subseteq \PP_n$ 
for $n \ge 4$ and several cubic surfaces 
in $\PP_3$.
We refer to~\cite{Bo2} for a sample use
of intrinsic quadrics as a testing class.

In the present article, we take a closer 
look at smooth intrinsic quadrics 
of small Picard number, but arbitrarily 
high dimension.
For toric varieties, the analogous idea 
has been pursued 
by Kleinschmidt~\cite{Kl} in Picard number two 
and by Batyrev~\cite{Ba} in Picard number three.
Moreover, in~\cite{FaHaNi}, we described 
all smooth, rational varieties of Picard 
number two that come with a
torus action of complexity one.
Similarly to the toric setting, where 
the restriction of being smooth of 
Picard number one allows just the 
projective spaces, the situation 
turns out to be simple for intrinsic 
quadrics: in Picard number 
one, we only find the classical smooth 
quadrics $X \subseteq \PP_n$, 
see Proposition~\ref{prop:picardnumber1}.
In Picard number two, we obtain a 
considerably larger class. 
The first main result of the paper provides
a full description of these varieties~$X$ 
in terms of their $\Cl(X)$-graded Cox ring and 
the semiample cone $\tau_X \subseteq \Cl_\QQ(X)$;
this collection of data indeed fixes $X$,
see Section~\ref{sec:quadrbasics} for a brief reminder
and~\cite{ArDeHaLa} for more background.

\begin{theorem}
\label{thm:smoothrhoX2}
Let $X$ be a smooth intrinsic quadric of 
Picard number $\rho(X)=2$.
Then $X$ has divisor class group 
$\Cl(X) \cong \ZZ^2$ 
and, with suitable integers $n,m \in \ZZ_{\ge 0}$,
the Cox ring of $X$ is given by
\begin{eqnarray*}
\mathcal{R}(X) 
& \cong & 
\KK[T_1,\ldots,T_n,S_1, \ldots, S_m] / \bangle{g},
\\[1ex]
g  
& = & 
\begin{cases}
T_1T_2 + \ldots + T_{n-1}T_n, & n \text{ even},
\\
T_1T_2 + \ldots + T_{n-2}T_{n-1} + T_n^2, & n \text{ odd}.
\end{cases}
\end{eqnarray*}
The possible constellations for the $\Cl(X)$-grading 
of $\mathcal{R}(X)$
and the semiample cone $\tau_X \subseteq \Cl_\QQ(X)$ are listed below;
we distinguish four types and write
$w_i \sei \deg(T_i)$ and $u_j \sei \deg(S_j)$ for the 
$\Cl(X)$-degrees.

\medskip
\noindent
\emph{Type~1:}
Fix $\alpha \in \ZZ_{\ge 0}$.
We have $n \ge 5$ and $m \ge 2$.
Moreover, $w_1 = \ldots = w_n = (1,0)$
and $u_j = (a_j,1)$ with 
$0 = a_1 \le a_2 \le \ldots \le a_m = \alpha$ holds.

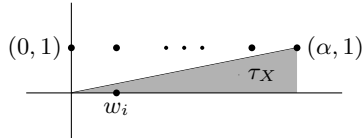
\begin{center}
    \begin{tikzpicture}[scale=0.6]
\draw[thin] (0,0)--(5,1);
\path[fill=gray!60!] (0,0)--(5,0)--(5,1)--(0,0);
\path[fill, color=black] (3.7,0.4) circle (0.0ex)  node[right]{\small $\tau_X$};
\draw (-1,0)--(6,0);
\draw (0,-1)--(0,2);
\path[fill, color=black] (1,0) circle (0.5ex)  node[below]{\small $w_i$};
\path[fill, color=black] (0,1) circle (0.5ex)  node[left]{\small $(0,1)$}; 
\path[fill, color=black] (1,1) circle (0.5ex)  node[]{};

\path[fill, color=black] (2.1,1) circle (0.3ex)  node[]{};
\path[fill, color=black] (2.5,1) circle (0.3ex)  node[]{};
\path[fill, color=black] (2.9,1) circle (0.3ex)  node[]{};

\path[fill, color=black] (4,1) circle (0.5ex)  node[]{};
\path[fill, color=black] (5,1) circle (0.5ex)  node[right]{\small $(\alpha,1)$};
\end{tikzpicture}   
\end{center}
Here, $X$ is the projectivization 
$\PP(\mathcal{O}_Y(a_1) \oplus \mathcal{O}_Y (a_2)
\oplus\ldots\oplus
\mathcal{O}_Y(a_m))$
of the split vector bundle 
defined by $a_1,\ldots,a_m$ 
over the smooth quadric 
$Y = V(g) \subseteq \PP_{n-1}$.

\medskip
\noindent
\emph{Type~2:} 
Fix $\alpha \in \ZZ_{\ge 0}$.
We have $n \ge 5$ and $m \ge 2$.
Moreover, $u_1 = \ldots = u_m = (1,0)$
holds and we have $w_i = (a_i,1)$ 
with $0 \le a_i \le \alpha$
for $i = 1,\ldots,n$ such that  
\begin{enumerate}
\item
$w_1=(0,1)$ and $w_2=(\alpha,1)$, 
\item
$w_i+w_{i+1} = (\alpha, 2)$ for all odd $i < n$ and 
$2w_n = (\alpha, 2)$, if $n$ is odd.
\end{enumerate}
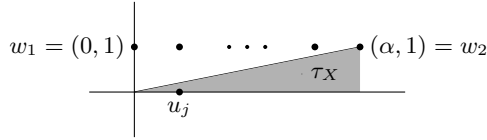
\begin{center}
    \begin{tikzpicture}[scale=0.6]
\draw[thin] (0,0)--(5,1);
\path[fill=gray!60!] (0,0)--(5,0)--(5,1)--(0,0);
\path[fill, color=black] (3.7,0.4) circle (0.0ex)  node[right]{\small $\tau_X$};
\draw (-1,0)--(6,0);
\draw (0,-1)--(0,2);
\path[fill, color=black] (1,0) circle (0.5ex)  node[below]{\small $u_j$};
\path[fill, color=black] (0,1) circle (0.5ex)  node[left]{\small $w_1=(0,1)$};
\path[fill, color=black] (1,1) circle (0.5ex)  node[]{};

\path[fill, color=black] (2.1,1) circle (0.3ex)  node[]{};
\path[fill, color=black] (2.5,1) circle (0.3ex)  node[]{};
\path[fill, color=black] (2.9,1) circle (0.3ex)  node[]{};

\path[fill, color=black] (4,1) circle (0.5ex)  node[]{};
\path[fill, color=black] (5,1) circle (0.5ex)  node[right]{\small $(\alpha,1)=w_2$};
\end{tikzpicture}   
\end{center}
Here, $X$ admits a locally trivial fibration 
$X \to \PP_{m-1}$ with fibers isomorphic to
the smooth quadric~$V(g) \subseteq \PP_{n-1}$.

\medskip
\noindent
{\em Type~3:} 
We have $n \ge 5$ and $m \ge 1$.
Moreover, $u_1 = \ldots = u_m = (1,0)$
holds and the $w_i$ satisfy
\begin{enumerate}
\item
$w_1 = (0,1)$ and $w_2 = (2,1)$,
\item
$w_2 = \ldots = w_{n-1} = (1,1)$.
\end{enumerate}
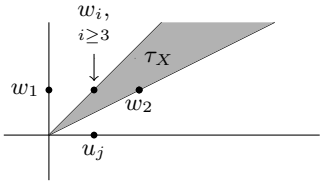
\begin{center}
    \begin{tikzpicture}[scale=0.6]
\draw (-1,0)--(6,0);
\draw (0,-1)--(0,2.5);
\draw[thin] (0,0)--(5,2.5);
\draw[thin] (0,0)--(2.5,2.5);
\path[fill=gray!60!] (0,0)--(2.5,2.5)--(5,2.5)--(0,0);
\path[fill, color=black] (1.9,1.7) circle (0.0ex)  node[right]{\small $\tau_X$};
\path[fill, color=black] (1,0) circle (0.5ex)  node[below]{\small $u_j$};
\path[fill, color=black] (2,1) circle (0.5ex)  node[below]{\small $w_2$};
\path[fill, color=black] (0,1) circle (0.5ex)  node[left]{\small $w_1$};
\path[fill, color=black] (1,1.7) circle (0ex)  node[above]{\small 
$\begin{array}{c}
w_i,\\[-.2em]
\scriptstyle{i \ge 3}
\end{array}$};
\draw[->] (1,1.8) -- (1,1.2);
\path[fill, color=black] (1,1) circle (0.5ex)  node[]{};
\path[fill, color=black] (2,1) circle (0.5ex)  node[]{};
\end{tikzpicture}   
\end{center}
Here, $X$ is the blowing-up of the projective space $\PP_{n+m-3}$ 
centered at the smooth quadric 
$V(g-T_1T_2, S_1,\ldots,S_m) \subseteq \PP_{n+m-3}$.

\medskip
\noindent
\emph{Type~4:} 
Fix $0 \le a \le \alpha \in \ZZ$.
We have $n \ge 6$ with $n$ even 
and $m \ge 0$. 
Moreover, $u_j = (a_j,1)$ holds
with $0 \le a_j \le \alpha$
and
\begin{enumerate}
\item
$w_1 = w_3 = \ldots = w_{n-1} = (1,0)$,
\item
$w_2 = w_4 = \ldots = w_n = (a,1)$,
\item
the vectors $(\alpha,1)$ and $(0,1)$ occur 
among $w_1,\ldots,w_n,u_1,\ldots,u_m$.
\end{enumerate}
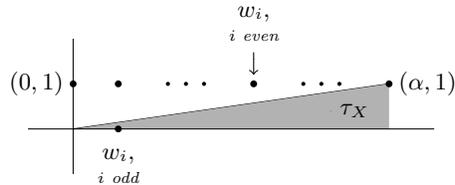
\begin{center}
\begin{tikzpicture}[scale=0.6]
\draw[thin] (0,0)--(7,1);
\path[fill=gray!60!] (0,0)--(7,0)--(7,1)--(0,0);
\path[fill, color=black] (5.7,0.4) circle (0.0ex)  node[right]{\small $\tau_X$};
\draw (-1,0)--(8,0);
\draw (0,-1)--(0,2);
\path[fill, color=black] (1,0) circle (0.5ex)  node[below]{\small $\begin{array}{c}
w_i,\\[-.2em]
\scriptstyle{i \text{ odd}}
\end{array}$};
\path[fill, color=black] (4,1.55) circle (0ex)  node[above]{\small 
$\begin{array}{c}
w_i,\\[-.2em]
\scriptstyle{i \text{ even}}
\end{array}$};
\draw[->] (4,1.7) -- (4,1.2);
\path[fill, color=black] (0,1) circle (0.5ex)  node[left]{\small $(0,1)$};
\path[fill, color=black] (1,1) circle (0.5ex)  node[]{};

\path[fill, color=black] (2.1,1) circle (0.3ex)  node[]{};
\path[fill, color=black] (2.5,1) circle (0.3ex)  node[]{};
\path[fill, color=black] (2.9,1) circle (0.3ex)  node[]{};

\path[fill, color=black] (4,1) circle (0.5ex)  node[]{};

\path[fill, color=black] (5.1,1) circle (0.3ex)  node[]{};
\path[fill, color=black] (5.5,1) circle (0.3ex)  node[]{};
\path[fill, color=black] (5.9,1) circle (0.3ex)  node[]{};

\path[fill, color=black] (7,1) circle (0.5ex)  node[right]{\small $(\alpha,1)$};
\end{tikzpicture}   
\end{center}
Here, $X$ admits a locally trivial fibration 
$X \to \PP_{n/2-1}$
with fibers isomorphic to the projective space~$\PP_{n/2+m-2}$.

\medskip
\noindent
Conversely, each of the above constellations in Types~1
to~4 defines a smooth intrinsic quadric of Picard number~2.
\end{theorem}

We say that an intrinsic quadric is~\emph{full} 
if all generators of its Cox ring show up
in the relation.
The full intrinsic quadrics of 
Theorem~\ref{thm:smoothrhoX2}
are precisely the cases of Type~4 with $m=0$
and hence $\alpha=0$;
they have already been found in~\cite{BeHa:2007}
under the additional hypothesis of 
a torsion free divisor class group.
Moreover, the cases $n = 5$ and $n = 6$ in Types~1
to~4 of Theorem~\ref{thm:smoothrhoX2} are precisely 
the smooth intrinsic quadrics allowing a torus 
action of complexity one and thus represent exactly 
the overlap with~\cite{FaHaNi}.

Recall that a normal projective variety $X$ 
is Fano if it admits an ample anticanonical 
divisor.
More generally, $X$ is called almost Fano if 
it has a numerically effective anticanonical divisor; 
we say that $X$ is truly almost Fano if it is 
almost Fano but not Fano.
Theorem~\ref{thm:smoothrhoX2} gives us in 
every dimension the (almost) Fano smooth 
intrinsic quadrics of Picard number two.

\begin{corollary}
\label{cor:FanosAlmFanos}
In the notation of Theorem~\ref{thm:smoothrhoX2},
the (truly almost) Fano varieties among the 
smooth intrinsic quadrics $X$ of Picard 
number two are characterized by the following 
conditions.
\begin{center}
\renewcommand{\arraystretch}{1.8} 
\begin{tabular}{c|c|c}
Type
& 
Fano 
& 
truly almost Fano
\\
\hline
1
&
$m\alpha < n-2 +a_1 + \ldots + a_m$
&
$m\alpha = n-2 +a_1 + \ldots + a_m$
\\
\hline
2
&
$\frac{n-2}{2}\alpha < m$
&
$\frac{n-2}{2}\alpha = m$
\\
\hline
3
&
$n-2>m$
&
$n-2=m$
\\
\hline
4
&
$
\begin{array}{ll}
m \alpha < \frac{n-2}{2} + a_1 + \ldots + a_m
\\[-2ex]
\text{and } w_2=(\alpha,1)
\end{array}
$
&
$
\begin{array}{ll}
m \alpha = \frac{n-2}{2} + a_1 + \ldots + a_m
\\[-2ex]
\text{and } w_2=(\alpha,1)
\end{array}
$
\\
\hline
4
&
&
$
\begin{array}{ll}
u_1=\ldots= u_m=(1,1)
\\[-2ex]
\text{and } w_2=(0,1)
\end{array}
$
\end{tabular}
\end{center}
\end{corollary}

Note that in Theorem~\ref{thm:smoothrhoX2},
the variety $X$ is of dimension $n+m-3$.
Thus, the above table provides us in particular
for every dimension with the numbers of 
(almost) Fano smooth intrinsic quadrics
of Picard number two.
The overlap with the classification of
smooth Fano threefolds by Mori and Mukai
consists of the threefold Type~3 with 
$n=5$, $m=1$, and the threefold of 
Type~4 with $n=6$, $m=0$, 
which occur as No.~2.30 and No.~2.32 
in~\cite{MoMu}, respectively.

For a Fano, not necessarily 
smooth, full intrinsic quadric $X$, 
we see in Proposition~\ref{prop:fullFano}
that its Picard number is bounded 
by $\rho(X) \le 3$.
Moreover, if $X$ is smooth, then we
can further show $\rho(X) \le 2$
and arrive at the following.

\begin{theorem}
\label{thm:fullFanos}
Let $X$ be a Fano smooth full intrinsic quadric.
Then $X$ is of Picard number $\rho(X) \le 2$ 
and
\begin{enumerate}
\item
if $\rho(X) = 1$ holds, then $X$ is isomorphic 
to the smooth projective quadric 
$V(T_0^2 + \ldots + T_n^2) \subseteq \PP_n$,
where $n \ge 4$,
\item
if $\rho(X) = 2$ holds, then $X$ 
is isomorphic to 
$V(T_0S_0 + \ldots + T_nS_n) \subseteq \PP_n \times \PP_n$,
the flag variety 
of type $(1,n-1,1)$, where $n \ge 2$.
\end{enumerate}
\end{theorem}

We use our results for testing Fujita's freeness
conjecture, which says that for any smooth 
projective variety $X$ with canonical 
divisor $C_X$, the divisor $C_X + sD$ 
is base point free provided that 
$D$ is ample and $s \ge \dim(X) + 1$ 
holds, see~\cite{Fu}.
This statement is known to hold
for varieties with torus 
action of complexity at most 
one~\cites{fujino,AlIl}
and in general up to dimension 
five~\cites{Re,EL1,Kaw,yezhu2}.
Corollary~\ref{cor:bpfquadricssatrhoX2}
verifies Fujita's freeness
conjecture for smooth intrinsic quadrics 
of Picard number at most two.

We turn to Picard number three.
Recall that smooth toric varieties of Picard number 
three have been described by Batyrev in~\cite{Ba}
in terms of primitive collections. 
In the setting of intrinsic quadrics, we obtain 
a complete picture in the full case.

\begin{theorem}
\label{thm:fullpic3}
Let $X$ be a full smooth intrinsic quadric of 
Picard number three.
Then $X$ has divisor class group $\Cl(X) = \ZZ^3$ 
and, with a suitable even integer $n \ge 8$,
the Cox ring of $X$ is given by
$$
\mathcal{R}(X) 
\ \cong \
\KK[T_1,\ldots,T_n] / \bangle{g}, 
\qquad\qquad 
g 
\ \sei \ 
T_1T_2 + \ldots + T_{n-1}T_n.
$$
The possible constellations for the $\Cl(X)$-gradings 
and the ample cone of $X$ are the following.
There is an integer $a\ge 0$ such that the degree of 
the relation is $\mu \sei \deg(g) = (1,a,1)$
and we have
$$
w_ 1 = w_3 = (0,1,0), 
\quad
w_2 = w_4 = (1,a-1,1),
$$
$$
w_5= (0,0,1),
\quad
w_6= (1,a,0),
\quad
w_7= (1,0,0),
\quad
w_8= (0,a,1).
$$
For all odd $9 \le i < n$, the generator
degrees $w_i$, $w_{i+1}$ coincide either with~$w_1,w_2$
or are located on the line segments~$\conv(w_5,w_8)$
and~$\conv(w_6,w_7)$.

\begin{center}
\begin{tikzpicture}[scale=1.3]


\coordinate (o) at (0,0);
\coordinate (w1) at (-0.5,-0.5);
\coordinate (s) at (-.5,-1.25); 
\coordinate (w2) at   ($(o)!.8!(s)$);
\coordinate (w5) at  (0,1);
\coordinate (w6) at (-.5,-2.5);
\coordinate (w7) at (-2.5,-1.25);
\coordinate (w8) at (1,0);
\coordinate (h1) at (-2.5,-2.5);
\coordinate (h2) at (-2,-1);
\coordinate (h3) at (-.75,-2);
\coordinate (t) at (-1.25,-2.25);

\draw[domain=-2.5:2,variable=\x] plot ({\x},{0});
\draw[domain=0:1,variable=\x] plot ({\x},{\x});
\draw[domain=-3:-2.5,variable=\x] plot ({\x},{\x});
\draw[domain=-.75:2.5,variable=\x] plot ({0},{\x});

\draw[dashed] (w5)--(w7);
\draw[dashed] (w6)--(w8);

\fill[gray!50,opacity=0.3] (o) -- (w7) -- (h1) -- cycle; 
\fill[gray!50,opacity=0.3] (o) -- (h1) -- (w6) -- cycle; 
\fill[gray!50,opacity=0.3] (o) -- (w6) -- (w7) -- cycle; 

\fill[gray!50,opacity=0.3] (o) -- (w6) -- (s) -- cycle; 
\fill[gray!50,opacity=0.3] (o) -- (h1) -- (s) -- cycle; 
\fill[gray!50,opacity=0.3] (o) -- (h1) -- (w6) -- cycle; 

\draw[gray!60] (o)--(s);
\draw[gray!60] (o)--(w6);
\draw[gray!60] (o)--(w7);
\draw[gray!60] (h1)--(w6)--(w7)--cycle;
\draw[gray!60] (h1)--(w6)--(s)--cycle;

\coordinate (c) at
(intersection cs: first line={(w7) -- (w6)},
                  second line={(h1) -- (s)});

\draw[thick,black!100,opacity=0.3] (o)--(c);
\draw[thick,black!100,opacity=0.3] (o)--(h1);
\draw[thick,black!100,opacity=0.3] (o)--(w6);
\draw[thick,black!100,opacity=100,rounded corners=0.1pt] (h1)--(w6)--(c)--cycle;
\path[fill, color=black] (-1.5,-2.3) circle (0ex) node[]{\tiny{ $\tau_X$}};

\path[fill, color=gray!80] (w1) circle (0.3ex);
\path[fill, color=black] (w1) circle (0.0ex)  node[left]{\tiny $w_1 \!$};
\path[fill, color=black] (w2) circle (0.3ex) node[left]{\tiny $w_2 \!$};
\path[fill, color=black] (w5) circle (0.3ex) node[left]{\tiny $w_5$};
\path[fill, color=black] (w6) circle (0.3ex) node[below]{\tiny $\qquad w_6$};
\path[fill, color=black] (w7) circle (0.3ex) node[left]{\tiny $w_8$};
\path[fill, color=black] (w8) circle (0.3ex) node[below]{\tiny $\qquad w_7$};
\end{tikzpicture}   
\end{center}

\noindent
Moreover, as indicated above, the semiample cone  
$\tau_X \subseteq \Cl(X)_{\QQ}$ 
of $X$ is given as the intersection of two cones:
$$ 
\tau_X
\ = \ 
\cone(w_1,w_2,w_6) 
\cap 
\cone(w_1,w_6,w_8). 
$$
Conversely, each of the above constellations defines a 
full smooth intrinsic quadric~$X$ of Picard number~3.
Moreover, each such~$X$ 
admits a locally trivial fibration with fibers 
a projective space onto a smooth projective toric 
variety of Picard number~two.
\end{theorem}

Again, we use this description 
to verify Fujita's freeness
conjecture also for full smooth intrinsic 
quadrics of Picard number three,
see Corollary~\ref{cor:fujitapic3full}.
As soon as we leave the full case, 
the situation in Picard number 
three becomes much more ample;
complete descriptions in the 
dimensions three and four have 
been elaborated in~\cite{Fa}.

\tableofcontents

\section{Basics on intrinsic quadrics}
\label{sec:quadrbasics}

We first discuss purely quadratic 
polynomials in general and present
a graded normal form in
Proposition~\ref{prop:quaddiag}.
Then we provide a quick guide  
to the general combinatorial 
theory of~\cite{ArDeHaLa}*{Chap.~3}
adapted to 
the sample class of intrinsic quadrics. 
This allows us in particular to encode 
and read off the necessary geometric 
properties.

Throughout the whole article, we work over 
an algebraically closed field $\KK$ of 
characteristic zero.
A grading of a $\KK$-algebra $R$ by a finitely 
generated abelian group $K$ is a direct sum
decomposition 
$$ 
R \ = \ \bigoplus_{w \in K} R_w
$$
into vector subspaces $R_w \subseteq R$ 
being compatible with multiplication in the 
sense that $R_wR_{w'} \subseteq R_{w+w'}$ 
holds for all $w,w' \in K$. 
A homomorphism of graded algebras $R = \oplus_K R_w$ 
and $S = \oplus_L S_u$ is a pair $(\psi,F)$ 
consisting of an algebra homomorphism 
$\psi \colon R \to S$ and a group homomorphism 
$F \colon K \to L$ such that one always has 
$\psi(R_w) \subseteq S_{F(w)}$.
In this situation, we speak of a graded homomorphism 
if $K = L$ holds and $F$ is the identity map.

\begin{proposition}
\label{prop:quaddiag}
Let $K$ be a finitely generated abelian group, 
consider a $K$-grading on the polynomial ring
$\KK[T_1,\ldots,T_s]$ such that the 
variables $T_1,\ldots,T_s$ and the following 
quadratic polynomial are $K$-homogeneous:
$$ 
g  
\ = \ 
\sum_{1 \le i \le j \le s} a_{ij} T_iT_j
\ \in \
\KK[T_1,\ldots,T_s].
$$
Then there are a linear 
automorphism 
$\psi \colon \lin(T_1,\ldots,T_s) \to \lin(T_1,\ldots,T_s)$ 
inducing a graded automorphism  
$\Psi \colon \KK[T_1,\ldots,T_s] \to \KK[T_1,\ldots,T_s]$ 
and non-negative integers 
$q,t$ with $q+t \le s$ such that
$$ 
\Psi(g) 
\ = \ 
g_{q,t}
\ \sei \
T_1T_2 + \ldots + T_{q-1}T_q + T_{q+1}^2 + \ldots + T_{q+t}^2
$$ 
and $\deg(T_{q+k}) \ne \deg(T_{q+l})$ holds for all  
$0 < k < l < t$.
In this setting, $s - q - t$ is the dimension of 
the singular locus of $V(g_{q,t}) \subseteq \KK^s$
and $t$ counts the $u \in K$ with $2u = \deg(g_{q,t})$ 
such that the number of $T_i$ of degree $u$ 
showing up in $g_{q,t}$ is odd.
\end{proposition}

\begin{proof}
Suitably renumbering the variables,
we may assume that $T_1, \ldots, T_r$
are precisely the variables that show 
up in $g$.
Let $w_1, \ldots, w_n \in K$ be the degrees 
of $T_1, \ldots, T_r$;
we impose $w_k \ne w_l$ for $k \ne l$ here.
Moreover, set $\mu \sei \deg(g) \in K$.
Further suitable renumbering of variables 
yields
$$ 
w_1 + w_2 = \ldots = w_{m}+w_{m+1} = \mu,
\qquad\qquad
2w_{m+2} = \ldots = 2w_n = \mu
$$
with a unique odd number $-1 \le m < n$.
Some of the variables $T_1, \ldots, T_s$ may 
share the same degree and we have 
$$ 
V 
\ \sei \ 
\lin(T_1,\ldots,T_s)
\ = \ 
V_1 \oplus \ldots \oplus V_n \oplus V_0,
$$
where $V_k$ is the linear subspace generated
by all $T_i$, $1 \le i \le r$, of degree $w_k$,
and $V_0$ is the linear subspace generated by 
the variables $T_{r+1}, \ldots, T_s$.
Suitably renumbering the $T_i$ again, we 
achieve
$$ 
T_1, \ldots, T_{d_1} \in V_1,
\quad
\ldots, 
\quad
T_{d_{n-1}+1}, \ldots, T_{d_n} \in V_n,
\quad
T_{d_n+1}, \ldots, T_s \in V_0.
$$

The idea is to build up $\psi$ stepwise
from appropriate endomorphisms $V \to V$.
First, consider variables $T_i \in V_1$ and 
$T_j \in V_2$ with $\alpha_{ij} \ne 0$.
Define a linear automorphism 
\begin{eqnarray*}
\psi_{ij} \colon V \ \to \ V,
\qquad
T_j \ \mapsto \ a_{ij}^{-1}T_j - a_{ij}^{-1}\sum_{k \ne j} a_{ik}T_k,
\quad
T_l \ \mapsto \ T_l \text{ for } l \ne j.
\end{eqnarray*}
Then $\psi_{ij}$ respects the direct 
sum decomposition of $V$ and restricts 
to the identity on all components different 
from $V_2$.
Moreover, $\psi_{ij}$ extends to an 
automorphism~$\Psi_{ij}$ of the $K$-graded 
algebra $\KK[T_1, \ldots, T_s]$ and we have 
$$ 
\Psi_{ij}(g) 
\ = \ 
\left(T_i+ \sum_{k \ne i} a_{ij}^{-1}a_{kj}T_k \right)T_j 
+ 
\sum_{k \ne i, \ l \ne j} \tilde{a}_{kl}T_kT_l
$$
with some $\tilde{a}_{kl}\in\KK$.
Now define a linear automorphism 
\begin{eqnarray*}
\psi_{ji} \colon V \ \to \ V,
\qquad
T_i \ \mapsto \ T_i - a_{ij}^{-1} \sum_{k \ne i}  a_{kj}T_k,
\quad
T_l \ \mapsto \ T_l \text{ for } l \ne i.
\end{eqnarray*}
Similarly as before, $\psi_{ji}$ respects the 
direct sum decomposition of $V$ and restricts 
to the identity on all components different 
from $V_1$.
Again, $\psi_{ji}$ extends to an 
automorphism~$\Psi_{ji}$ of the $K$-graded 
algebra $\KK[T_1, \ldots, T_s]$.
This time we have 
$$ 
\Psi_{ji}(\Psi_{ij}(g)) 
\ = \ 
T_iT_j + \sum_{k \ne i, \ l \ne j} \tilde{a}_{kl}T_kT_l.
$$
Thus, a suitable composition of the 
automorphisms $\Psi_{ji} \circ \Psi_{ij}$ 
turns $g$ into the desired form with respect 
to the variables from $V_1$ and $V_2$.
Proceeding similarly, we can settle all 
other pairs $V_{l}$ and $V_{l+1}$ for 
$l = 3, 5, \ldots, m$.

On each subspace $V_k$ for $k > m+1$, 
the variables all have the same $K$-degree
and, if a variable of a given monomial of $g$ 
belongs to $V_k$, then all variables of 
this monomial belong to $V_k$.
Thus, we may treat the part $q_k$ of $q$ built
from variables of $V_k$ separately.  
The usual diagonalization procedure 
for the Gram matrix of $q_k$ leads to 
a presentation of $q_k$ as a sum of squares.
If the number $c_k$ of these squares is even,
then we turn the whole $q_k$ into a sum of 
terms $T_iT_j$ with $i \ne j$.
Otherwise, we turn $q_k$ into a  
sum of $T_iT_j$ with $i \ne j$ 
plus one single square.
\end{proof}

We call $g_{q,t} \in \KK[T_1,\ldots,T_s]$ 
as in Proposition~\ref{prop:quaddiag} a 
\emph{standard $K$-homogeneous quadratic polynomial}.
As the supplement of the proposition shows, 
a given standard $K$-homogeneous quadratic 
polynomial $g_{q,t}$
can be transformed via an automorphism 
of graded algebras into another one,
say $g_{q',t'}$, if and only if $q=q'$ 
and $t=t'$ hold.
If for some $g_{q,t}$ the sum $q+t$ is odd, 
then we must have $t \ge 1$.
Let us briefly discuss what happens if 
$t > 1$ holds.

\begin{remark}
Let $g_{q,t} \in \KK[T_1,\ldots,T_s]$ 
be a standard $K$-homogeneous quadratic
polynomial with $t > 1$.
Then, for any two $1 \le i < j \le t$,
twice the degree of $T_{q+i}$ as well as 
twice the degree of $T_{q+j}$ equal the 
degree of $g_{q,t}$ and thus we have
$$ 
2(\deg(T_{q+i}) - \deg(T_{q+j})) \ = \ 0 \ \in \ K.
$$
In particular, the number $t$ is bounded by 
the order of the subgroup $K_2 \subseteq K$ 
consisting of all elements annihilated by 
multiplication with $2$.    
Here is a concrete example: Take 
$$
K \ = \ \ZZ/2\ZZ \times \ZZ/2\ZZ,
\qquad\qquad
g \ = \ T_1^2 + T_2^2 \ \in \ \KK[T_1,T_2].
$$ 
Define a $K$-grading on $\KK[T_1,T_2]$ 
by  setting
$\deg(T_1) \sei (\bar{1},\bar{0})$ and 
$\deg(T_2) \sei (\bar{0},\bar{1})$.
Then $g = g_{0,2}$ is a standard $K$-homogeneous 
quadratic polynomial in $\KK[T_1,T_2]$.
\end{remark}

We turn to the construction of intrinsic 
quadrics. 
Recall that every Mori dream space, that means, 
every irreducible, normal, projective variety 
$X$ with finitely generated divisor class 
group $\Cl(X)$ and finitely generated Cox ring 
$$ 
\mathcal{R}(X)
\ = \ 
\bigoplus_{\Cl(X)} \Gamma(X,\mathcal{O}_X(D))
$$
can be retrieved from $\mathcal{R}(X)$ as follows.
The above grading defines an action 
of the quasitorus $H = \Spec \; \KK[\Cl(X)]$ on the 
total coordinate space 
$\overline{X} = \Spec \, \mathcal{R}(X)$.
If $u \in \Cl(X)$ is any ample 
class of $X$, then the associated set 
of semistable points is  
$$
\overline{X}^{ss}(u)
\ = \ 
\{
x \in \overline{X}; \; 
f(x) \ne 0 
\text{ for some } 
f \in \mathcal{R}(X)_{nu}, 
\text{ where } n > 0
\}
\ \subseteq  \
\overline{X}.
$$
This is an open $H$-invariant set and the 
variety $X$ is obtained as the associated 
geometric invariant theory quotient 
$X = \overline{X}^{ss}(u) \quot H$.
We refer to~\cite{ArDeHaLa} for more 
background.

Reversing the picture just drawn, we can produce 
all Mori dream spaces from suitable finitely 
generated, normal, integral, $K$-graded $\KK$-algebras
$$ 
R \ = \ \bigoplus_{w \in K} R_w,
$$
where ``suitable'' characterizes the Cox rings
among these algebras.
Let us briefly recall from~\cite{ArDeHaLa}
what that means.
First, $R$ has to be \emph{$K$-factorial} in the sense 
that we have unique factorization in the multiplicative 
monoid $R_\times \subseteq R$ of non-zero homogeneous 
elements of $R$; 
for instance, $R$ can be a unique factorization domain 
in the classical sense.
For the further conditions, fix any system 
$f_1, \ldots, f_s$  of pairwise non-associated $K$-prime,
i.e., prime in $R_\times$, generators of 
$R$ and consider the (convex, polyhedral) cones
$$ 
\kappa_0 
\ \sei \ 
\cone(\deg(f_1), \ldots, \deg(f_s)),
\qquad
\kappa_1
\ \sei \ 
\bigcap_{i = 1}^s \cone(\deg(f_j); \ j \ne i)
$$
in the rational vector space $K_\QQ  =K \otimes_\ZZ \QQ$ 
associated with $K$.
Then we ask the $K$-grading to be \emph{pointed} 
in the sense that $R_0 = \KK$ holds and the 
\emph{weight cone} $\kappa_0$ contains no lines.
Moreover, the $K$-grading must be \emph{almost free}
in the sense that any $s-1$ of the $\deg(f_i)$ 
generate $K$ as a group.
Finally, the \emph{moving cone} $\kappa_1$ has to be 
of full dimension in $K_\QQ$.

\begin{example}
\label{ex:quadkfact}
Let a finitely generated abelian group $K$ and 
a pointed, almost free $K$-grading of the polynomial 
ring $\KK[T_1, \ldots, T_s]$
be given such that all variables~$T_i$ are $K$-homogeneous
and the moving cone is of full dimension in $K_\QQ$.
Moreover, let $g_{q,t} \in \KK[T_1, \ldots, T_s]$ be
a standard $K$-homogeneous quadratic polynomial
and consider the 
factor algebra $R  \sei  \KK[T_1, \ldots, T_s] / \bangle{g_{q,t}}$
with its induced $K$-grading.
\begin{enumerate}
\item
If $q+t \ge 5$ holds, then $R$ is a unique factorization 
domain and the $K$-grading of $R$ is factorial.
\item
For $q+t < 5$, the ring is normal, integral with
factorial $K$-grading if and only if $K = \ZZ^s / M$ 
and $\deg(T_i) = e_i + M$ hold, where $M$ is the row space 
of an $r \times s$ matrix with $r < s$ of the 
following shape
\begin{eqnarray*}
q = 0, \ t = 4:
& & 
{
\small
\left[
\begin{array}{rrrrrrr}
-2 & 2 & 0 & 0  & 0 & \ldots & 0
\\
-2 & 0 & 2 & 0  & 0 & \ldots & 0
\\
-2 & 0 & 0 & 2  & 0 & \ldots & 0
\\
d_1 & d_2 & d_3 & d_4 & d_1' & \ldots & d_m' 
\end{array}
\right],
}
\\[1ex]
q=2, \ t=2:
& &
{
\small
\left[
\begin{array}{rrrrrrr}
-1 & -1  & 2 & 0 & 0  & \ldots & 0
\\
-1 & -1 & 0 & 2 & 0  & \ldots & 0
\\
d_1 & d_2 & d_3 & d_4 & d_1' & \ldots & d_m' 
\end{array}
\right],
}
\\[1ex]
q=0, \ t=3:
& & 
{
\small
\left[
\begin{array}{rrrrrr}
-2 & 2 & 0 & 0  & \ldots & 0
\\
-2 & 0 & 2 & 0  & \ldots & 0
\\
d_1 & d_2 & d_3 & d_1' & \ldots & d_m' 
\end{array}
\right].
}
\end{eqnarray*}
\end{enumerate}
In Case~(ii), the conditions ``almost free'', ``pointed'' 
and ``full-dimensional moving cone'' on the $K$-grading 
mean that the columns of the listed matrices are pairwise 
different primitive lattice points in $\ZZ^r$ generating
$\QQ^r$ as a cone.
For the last two cases, the statement on $K$-factoriality
follows from the results of~\cite{HaHe} and for the first 
one, a proof in a more general framework will be presented
elsewhere.
\end{example}

We are ready for explicitly constructing 
intrinsic quadrics.
The notation introduced in the subsequent two 
constructions will be used thoughout the whole 
article.

\begin{construction}[Standard intrinsic quadrics]
\label{constr:intquad}
Consider a pointed $K$-grading of the polynomial ring 
$\KK[T_1, \ldots, T_n,S_1,\ldots,S_m]$,
where $K$ denotes a finitely generated abelian 
group
and where all variables $T_i$ and $S_j$ 
are $K$-homogeneous,
any $n+m-1$ of their degrees generate $K$ as a 
group and the moving cone is of full dimension 
in $K_\QQ$.
Moreover, let 
$$
g_{q,t} \ \in \ \KK[T_1, \ldots, T_n,S_1,\ldots, S_m]
$$ 
be a standard $K$-homogeneous quadratic polynomial 
with $3 \le q+t = n$;
thus, by choice of notation, $g_{q,t}$ depends 
precisely on the variables $T_1,\ldots, T_n$.
Assume that the $K$-grading is factorial,
that means that
Condition~\ref{ex:quadkfact}~(i) or~(ii) 
is satisfied.
Take any $u \in K$ from the relative interior 
of the moving cone.
Then we obtain a commutative diagram
$$ 
\xymatrix{
V(g_{q,t}) 
\ar@{}[r]|=
&
{\overline{X}}
\ar@{}[r]|\subseteq
\ar@{}[d]|{\rotatebox[origin=c]{90}{$\scriptstyle\subseteq$}}
&
{\overline{Z}}
\ar@{}[r]|=
\ar@{}[d]|{\rotatebox[origin=c]{90}{$\scriptstyle\subseteq$}}
&
{\KK^{n+m}}
\\
&
{\overline{X}}^{ss}(u)
\ar@{}[r]|\subseteq
\ar[d]_{\quot H}
&
{\overline{Z}}^{ss}(u)
\ar[d]^{\quot H}
&
\\
&
X
\ar[r]
&
Z
}
$$
where $H = \Spec \, \KK[K]$ is the quasitorus 
corresponding to $K$, the downwards arrows 
are the GIT-quotients defined by $u$ and 
the bottom horizontal arrow is a closed
embedding.
Moreover, $X$ and $Z$ are normal projective 
varieties and we have 
$$ 
\dim(X) +1 
\ = \  
\dim(Z)
\ = \ 
n+m - \dim(K_\QQ),
\qquad 
\Cl(X) \ = \ \Cl(Z) \ = \ K
$$
for the respective dimensions and 
divisor class groups. 
Moreover, $Z$ is a toric variety and 
we call $X = X(q,t,m,u)$ a \emph{standard intrinsic quadric}.
The Cox ring of $X$ is given as $K$-graded factor 
algebra
$$
\mathcal{R}(X) 
\ = \
\KK[T_1, \ldots, T_n,S_1,\ldots,S_m] / \bangle{g_{q,t}}.
$$ 
By a \emph{full intrinsic quadric}, we mean an
intrinsic quadric with a defining quadratic 
polynomial $g$ such that the normal form of $g$ 
is $g_{q,t} \in \KK[T_1,\ldots,T_n]$ with $n = q+t$,
that means that there are no free variables~$S_j$.
\end{construction}

By definition, intrinsic quadrics are normal 
projective varieties $X$ admitting a presentation
of the Cox ring by $\Cl(X)$-homogeneous generators
such that the ideal of relations is generated 
by single, purely quadratic, 
$\Cl(X)$-homogeneous polynomial.
As an immediate consequence of 
Proposition~\ref{prop:quaddiag} we obtain
the following.

\begin{proposition}
\label{prop:iq2siq}
Every intrinsic quadric is isomorphic to a 
standard intrinsic quadric.
\end{proposition}

Note that in Construction,
the set of semistable points 
${\overline{Z}}^{ss}(u)$ is an open toric 
subvariety of ${\overline{Z}} = \KK^{n+m}$
and the quotient map
$\pi \colon {\overline{Z}}^{ss}(u) \to Z$ 
for the action of $H$ is a toric morphism;
in fact this is the usual quotient 
presentation of the toric variety $Z$
from~\cite{Cox}.
Cutting down the orbit decomposition
from the ambient toric variety $Z$ to $X$
yields a decomposition of $X$ into locally 
closed subvarieties which we call
the \emph{pieces} of $X$.
We need to identify these pieces 
explicitly.

\begin{construction}
\label{constr:stratif}
Notation as in Construction~\ref{constr:intquad}.
The  degree homomorphism $Q \colon \ZZ^{n+m} \to K$ sending 
the $i$-th canonical basis vector $e_i \in \ZZ^{n+m}$
to the weight $\deg(T_i) \in K$ gives rise to a pair of mutually
dual exact sequences of abelian groups
$$ 
\xymatrix{
0 
\ar[r]
&
L
\ar[r]
&
{\ZZ^{n+m}}
\ar[r]^{P}
&
{\ZZ^r}
&
\\
0 
\ar@{<-}[r]
&
K
\ar@{<-}[r]_{\!\!\! Q}
&
{\ZZ^{n+m}}
\ar@{<-}[r]_{ \textcolor{white}{sss} P^*}
&
{\ZZ^r}
\ar@{<-}[r]
&
0
}
$$
For every face $\gamma_0 \preceq \gamma$
of the positive orthant $\gamma = \QQ_{\ge 0}^{n+m}$,
denote by 
$\overline{Z}(\gamma_0) \subseteq \overline{Z}$
the set of all points $z \in \overline{Z}$ 
having coordinates $z_i \ne 0$ if $e_i \in \gamma_0$ 
and $z_i = 0$ otherwise.
This sets up a bijection  
$$
\{\text{faces of } \gamma\}
\ \to \
\{\text{toric orbits of } \overline{Z} \},
\qquad
\gamma_0 
\ \mapsto \
\overline{Z}(\gamma_0).
$$
A face $\gamma_0 \preceq \gamma$ 
is called \emph{$Z$-relevant}, if the cone 
$Q(\gamma_0) \subseteq K_\QQ$  
contains $u$ in its relative interior.
The set of semistable points 
$\overline{Z}^{ss}(u)$ is the union
of all toric orbits $\overline{Z}(\gamma_1)$,
where $\gamma_0 \preceq \gamma_1$ 
with a $Z$-relevant $\gamma_0 \preceq \gamma$.
Via the quotient map
$\pi \colon {\overline{Z}}^{ss}(u) \to Z$,
we obtain a bijection
$$
\{Z \text{-relevant faces of } \gamma\}
\ \to \
\{\text{toric orbits of } Z\},
\qquad
\gamma_0 
\ \mapsto \
Z(\gamma_0) \sei \pi(\overline{Z}(\gamma_0)).
$$
We say that $\gamma_0 \preceq \gamma$ is an
\emph{$\overline{X}$-face} if
$\overline{X}(\gamma_0) \sei \overline{X} \cap \overline{Z}(\gamma_0)$
is non-empty and we call it \emph{$X$-relevant}
if in addition $\gamma_0$ is $Z$-relevant.
The $X$-relevant faces of $\gamma$
correspond to the toric orbits of $Z$ intersecting~$X$ 
non-trivially. This leads to a bijection
$$
\{X \text{-relevant faces of } \gamma\}
\ \to \
\{\text{pieces of } X\},
\qquad
\gamma_0 \ \mapsto \ X(\gamma_0) \sei X \cap Z(\gamma_0).
$$
The \emph{covering collection} of $X$ is the set 
$\cov(X)$ of all minimal $X$-relevant faces of $\gamma$. 
The union over all affine toric charts 
$Z_{\gamma_0} \subseteq Z$,
where $\gamma_0$ stems from the covering collection,
is the \emph{minimal toric ambient variety} of $X$; 
it is the minimal open toric subvariety of $Z$
containing $X$ as a closed subvariety. 
\end{construction}

\begin{remark}
\label{rem:Xfaces}
Due to the specific form of $g_{q,t} \in \KK[T_1,\ldots,T_n,S_1, \ldots S_m]$, 
we can explicitly 
describe the faces $\gamma_0 \preceq \gamma$ defining 
a non-empty set
$\overline{X}(\gamma_0) = \overline{X} \cap \overline{Z}(\gamma_0)$.
For any sequence $1 \le i_1 < \ldots < i_k \le n+m$, 
we denote
$$ 
\gamma_{i_1,\ldots,i_k}
\ \sei \ 
\cone(e_{i_1}, \ldots, e_{i_k}) 
\ \preceq \
\gamma.
$$ 
This gives us all the faces of the orthant 
$\gamma = \QQ_{\ge 0}^{n+m}$.
We consider the following four basic types of 
faces:
\begin{enumerate}
\item
$\gamma_{i,i+1,j,j+1}$ with $1 \le i < j < q$ odd,
\item
$\gamma_{i,i+1,j}$ with $1 \le i < q$ odd and $q+1 \le j \le q+t$,
\item
$\gamma_{i,j}$ with $q+1 \le i < j \le q+t$,
\item
$\gamma_{i_1,\ldots,i_k,q+t+1, \ldots, q+t+m}$,
where $i_1 \in \{1,2\}, \ i_2 \in \{3,4\}, \ldots, 
\ i_k \in \{q-1,q\}$ with $k=q/2$. 
\end{enumerate}
Then, $\gamma_0 \preceq \gamma$ is an $\overline{X}$-face, 
i.e.~the set  $\overline{X}(\gamma_0)$ in non-empty,
if and only if one of the following holds
\begin{itemize}
\item 
$\tau \preceq \gamma_0$ with a face $\tau \preceq \gamma$ 
of type~(i), type~(ii) or type~(iii).
\item 
$\gamma_0 \preceq \tau$ with a face $\tau \preceq \gamma$ 
of type~(iv).
\end{itemize}
\end{remark}

A point $x \in X$ of a variety is \emph{factorial} 
if the local ring $\mathcal{O}_{X,x}$ is a unique 
factorization domain.
A variety $X$ is called \emph{locally factorial} 
if all its points are factorial; this is equivalent 
to the property that every Weil divisor of $X$ is Cartier.
We say that a standard intrinsic 
quadric~$X$ arising from Construction~\ref{prop:quaddiag}
is \emph{quasismooth} if $\overline{X}^{ss}(u)$
is smooth; this implies that $X$ has at most abelian 
quotient singularities.

\begin{proposition}
\label{prop:smooth}
Let $X = X(q,t,m,u)$ be a standard intrinsic quadric
arising from Construction~\ref{prop:quaddiag}.
\begin{enumerate}
\item
Let $\gamma_m \sei \cone(e_{q+t+1}, \ldots, e_{q+t+m}) \preceq \gamma$. 
Then the singular locus of the total coordinate 
space $\overline{X} = V(g_{q,t})$ is given by
$$ 
\overline{X}^{\mathrm{sing}}
\ = \ 
V(T_1,\ldots,T_{q+t})
\ = \ 
\bigcup_{\gamma_0 \preceq \gamma_m} \overline{X}(\gamma_0)
\ \subseteq \
\overline{X}. 
$$ 
\item 
The variety $X$ is quasismooth
if and only if every $X$-relevant face 
$\gamma_0 \preceq \gamma$
contains some $e_i$ with $1 \le i \le q+t$.
\item
The piece $X(\gamma_0)$ associated with 
an $X$-relevant 
$\gamma_0 \preceq \gamma$ consists of 
locally factorial points of $X$ 
if and only if $Q(\lin_\QQ(\gamma_0) \cap \ZZ^{n+m})$ 
generates $K$ as a group.
\item
The variety $X$ is locally factorial if and only if 
for every $\gamma_0 \in \cov(X)$, the image
$Q(\lin_\QQ(\gamma_0)\cap \ZZ^{n+m})$ generates $K$ as a group. 
\item
The piece $X(\gamma_0)$ associated with an $X$-relevant 
$\gamma_0 \preceq \gamma$ consists of 
smooth points of $X$ if and only if the 
following two statements hold: 
\begin{enumerate}
\item
$Q(\lin_\QQ(\gamma_0) \cap \ZZ^{n+m})$ generates $K$ as a group,
\item
$e_i \in \gamma_0$ holds for some $1 \le i \le q+t$.
\end{enumerate}
\item
The variety $X$ is smooth if and only if it is quasismooth
and for every $\gamma_0 \in \cov(X)$, the image
$Q(\lin_\QQ(\gamma_0) \cap \ZZ^{n+m})$
generates $K$ as a group. 
\end{enumerate}
\end{proposition}

\begin{proof}
The first statement is obvious and the remaining ones 
are the adapted versions 
of~\cite{ArDeHaLa}*{Cor.~3.3.1.8 and Prop.~3.3.1.10}. 
\end{proof}

The following three statements are proven 
in more generality in~\cite{ArDeHaLa}*{Sec.~3.3}.
Below, we denote for a convex, polyhedral cone 
$\sigma$ in a rational vector space $V$, its
relative interior by $\sigma^{\circ}$.

\begin{proposition}
\label{prop:picard}
Let $X = X(q,t,m,u)$ be a standard intrinsic quadric
arising from Construction~\ref{constr:intquad}.
Then the Picard group of $X$ is given as
$$ 
\Pic(X) 
\ = \ 
\bigcap_{\gamma_0 \in \cov(X)} Q(\lin_\ZZ(\gamma_0)\cap E)
\ \subseteq \ 
K
\ = \ 
\Cl(X). 
$$
\end{proposition}

\begin{proposition}
\label{prop:ample}
Let $X = X(q,t,m,u)$ be a standard intrinsic quadric
arising from Construction~\ref{constr:intquad}.
Then the cones of effective, movable, 
semiample and ample divisor classes 
of $X$ in $\Cl_\QQ(X) = K_\QQ$ are given as
$$ 
\Eff(X) \ = \ Q(\gamma),
\qquad
\Mov(X) 
\ = \ 
\bigcap_{\gamma_0 \preceq \gamma \text{ facet }} Q(\gamma_0),
$$
$$ 
\SAmple(X) 
\ = \ 
\bigcap_{\gamma_0 \in \cov(X)} Q(\gamma_0),
\qquad
\Ample(X) 
\ = \ 
\bigcap_{\gamma_0 \in \cov(X)} Q(\gamma_0)^{\circ}.
$$
Moreover, for every $u' \in \Ample(X)$ we have 
$\overline{X}^{ss}(u) = \overline{X}^{ss}(u')$ 
for the sets of semistable points and thus 
$X = X(q,t,m,u')$.
\end{proposition}

\begin{proposition}
\label{prop:Qfact}
Let $X = X(q,t,m,u)$ be a standard intrinsic quadric
arising from Construction~\ref{constr:intquad}.
Then the following statements are equivalent.
\begin{enumerate}
\item
$X$ is $\QQ$-factorial.
\item
For every $X$-relevant $\gamma_0 \preceq \gamma$
the image $Q(\gamma_0)$ is of full dimension in 
$K_\QQ$.
\item
The semiample cone $\SAmple(X)$ 
is of full dimension in 
$K_\QQ$.
\end{enumerate}
\end{proposition}

\section{Picard numbers one and two:~classification}
\label{sec:proof-thm-1-1}

First, we describe all locally 
factorial instrinsic quadrics
of Picard number one.
Then we show that locally 
factorial intrinsic quadrics
of Picard number two have 
torsion free divisor class group,
see Proposition~\ref{prop:PicXtorsfree}.
Finally, as the first part 
of the proof of Theorem~\ref{thm:smoothrhoX2},
we establish the normal forms 
for the smooth intrinsic quadrics given there.

\begin{proposition}
\label{prop:picardnumber1}
Let $X$ be a locally factorial intrinsic quadric 
of Picard number one.
Then $X$ has divisor class group 
$\Cl(X) \cong \ZZ$ 
and, with suitable integers $n \ge 5$ and 
$m \ge 0$,
the Cox ring of $X$ is given by
\begin{eqnarray*}
\mathcal{R}(X) 
& \cong & 
\KK[T_1,\ldots,T_n,S_1, \ldots, S_m] / \bangle{g},
\\[1ex]
g  
& = & 
\begin{cases}
T_1T_2 + \ldots + T_{n-1}T_n, & n \text{ even},
\\
T_1T_2 + \ldots + T_{n-2}T_{n-1} + T_n^2, & n \text{ odd}.
\end{cases}
\end{eqnarray*}
The $\Cl(X)$-grading of $\mathcal{R}(X)$ is 
given by $\deg(T_i) = \deg(S_j) = 1$ 
for all $i = 1, \ldots, n$ and $j= 1, \ldots, m$.
Thus, $X$ is isomorphic to the 
classical quadric $V(g) \subseteq \PP^{n+m-1}$
with singular locus $V(T_1,\ldots,T_n)$.
In particular, $X$ is smooth if and only if 
$m=0$ holds.
\end{proposition}

\begin{proof}
Since $X$ is locally factorial, we have 
$\Pic(X) = \Cl(X)$.
In particular, $\Cl_{\QQ}(X)$ is of dimension
one.  
We may assume that $X$ arises from 
Construction~\ref{constr:intquad} with a 
standard $\Cl(X)$-homogeneous quadratic 
polynomial $g_{q,t}$ 
and that the ample cone $\Ample(X)$ is the 
positive ray in $\Cl_{\QQ}(X) = \QQ$.
Consider the faces 
$$
\gamma_0 \ \sei \ \cone(e_i), 
\qquad
i = 1 ,\ldots,q
\quad\text{or}\quad  
i = n+1, \ldots, n+m.
$$
Each of these faces is $X$-relevant.
Since $X$ is locally factorial, 
$Q(e_i)$ generates $\Cl(X)$ as a group,
see Proposition~\ref{prop:smooth}.
In particular, if $q \ge 2$ holds, 
then we can conclude $\Cl(X) = \ZZ$
and 
$$
\deg(T_i) \ =  \ 1,
\quad 
i = 1, \ldots, n,
\qquad
\deg(S_j) \ =  \ 1,
\quad 
j = 1, \ldots, m.
$$
This implies $t \le 1$.
The cases $n=3,4$ are impossible:
then the Cox ring $\mathcal{R}(X)$ 
wouldn't admit unique factorization, 
but it has to do so because of the torsion free 
divisor class group~$\Cl(X)$, 
see~\cite{ArDeHaLa}*{Prop.~1.4.1.5}.
Thus, we also have $n \ge 5$, if  
$q \ge 2$ holds.

We exclude the case $q=0$.
Here, $t \ge 3$ must hold.
Thus, we have the $X$-relevant face 
$\gamma_0 = \cone(e_{1},e_{2})$.
Thus, $\Cl(X)$ is generated by 
$\deg(T_{1})$ and $\deg(T_{2})$
which implies 
$\Cl(X) = \ZZ \oplus \Gamma$ with a cyclic 
group $\Gamma = \ZZ / k \ZZ$ 
and, after applying a suitable automorphism
of $\Cl(X)$, we may assume
$$
\deg(T_{1}) \ = \ (1,\overline{0}),
\qquad
\deg(T_{2}) \ = \ (1,\overline{1}).
$$
Since $2 \deg(T_{1}) = 2 \deg(T_{2})$
holds, we obtain $k = 2$.
But then there is no way to assign 
to $T_3$ a degree in $\Cl(X)$ differing 
from the degrees of $T_1$ and $T_2$;
a contradiction. 
\end{proof}

\begin{remark}
\label{rem:effrho2}
Let $X$ be a $\QQ$-factorial 
standard intrinsic quadric
with $\Cl_\QQ(X)$ of dimension two
arising from 
Construction~\ref{constr:intquad}.
Then the effective cone $\Eff(X)$ is 
uniquely decomposed into three convex sets
$$
\Eff(X) 
\ = \ 
\tp \cup \tau_X^\circ \cup \tm,
$$
such that~$\tp$ and~$\tm$ do not intersect
$\tau_X^\circ = \Ample(X)$ and 
$\tp \cap \tm$ consists of the origin.
Because of $\tau_X^\circ \subseteq \Mov(X)$,
each of $\tp$ and~$\tm$ 
contains at least two (not necessarily different) 
degrees of the Cox ring generators 
$T_1,\ldots,T_n,S_1,\ldots,S_m$.
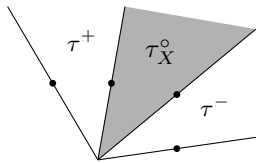
\begin{center}
    \begin{tikzpicture}[scale=0.6]
    \path[fill=gray!60!] (0,0)--(3.5,2.9)--(0.6,3.4)--(0,0);
    \path[fill, color=black] (1.75,1.45) circle (0.5ex)  node[]{};
    \path[fill, color=black] (1.4,2.4) circle (0.0ex)  node[]{$\tau_X^\circ$};
     \path[fill, color=black] (0.3,1.7) circle (0.5ex)  node[]{};
    \draw (0,0)--(0.6,3.4);
    \draw (0,0) --(-2,3.4);
  \path[fill, color=black] (-1,1.7) circle (0.5ex)  node[]{};
    \path[fill, color=black] (-0.35,2.65) circle (0.0ex)  node[]{\small{$\tau^+$}};
    \draw (0,0)  -- (3.5,2.9);
    \draw (0,0)  -- (3.5,0.5);
  \path[fill, color=black] (1.75,0.25) circle (0.5ex)  node[]{};
    \path[fill, color=black] (2.6,1.2) circle (0.0ex)  node[]{\small{$\tau^-$}};
    \path[fill, color=white] (4,1.9) circle (0.0ex);
  \end{tikzpicture}   
\end{center}
Note that $\tau_X^\circ$ is an open cone of dimension two,
whereas $\tau^-$ as well as $\tau^+$ might be 
one-dimensional.
The closure $\tau_X = \SAmple(X)$ 
of $\tau_X^\circ$ is the 
intersection of two $X$-relevant faces,
see Proposition~\ref{prop:ample},
and thus we find degrees of variables on
its boundary.
Moreover, apart from $\deg(T_n)$ when $t=1$,
no degree of a $T_i$ or a $S_j$ can lie 
in $\tau_X^\circ$, use again 
Proposition~\ref{prop:ample}.
\end{remark}

\begin{proposition}
\label{prop:PicXtorsfree}
Let $X$ be an intrinsic quadric of Picard number two.
If $X$ is locally factorial, 
then $\Cl(X) = \Pic(X) = \ZZ^2$ holds. 
\end{proposition}

\begin{proof}
Since $X$ is locally factorial, every Weil divisor 
is principal and thus we have $\Cl(X) = \Pic(X)$. 
The remaining task is to show that $\Pic(X)$ is 
torsion free.
For this, we may assume that $X$ arises from 
Construction~\ref{constr:intquad}.
We claim that it suffices to find a two-dimensional
$X$-relevant face $\gamma_0 \preceq \gamma$.
Indeed, Proposition~\ref{prop:picard} tells us
that $\Pic(X)$ is a subgroup of
$Q(\lin_\QQ(\gamma_0) \cap \ZZ^{n+m}) \subseteq \Cl(X)$.
In particular, $Q(\lin_\QQ(\gamma_0))$ is of dimension 
two.
Consequently, being generated by two elements, 
$Q(\lin_\QQ(\gamma_0) \cap \ZZ^{n+m})$ is torsion free.
Then also $\Pic(X)$ must be torsion free.

Now, the ample cone $\tau_X^\circ$ is two-dimensional
and, according to Remark~\ref{rem:effrho2},
we find two degrees $v_1^-,v_2^- \in \tau^-$
and two degrees $v_1^+,v_2^+ \in \tau^+$ 
stemming from, in total, four of the generators
$T_i$ and $S_j$ such that $v_1^-$ and $v_1^+$
generate the effective cone $\Eff(X)$.
After suitably renumbering the $T_i$ and the $T_j$,
we are in one of the following cases:

\medskip
\noindent
\emph{Case 1.}
We have $\deg(S_1) \in \tau^-$ 
and $\deg(S_2) \in \tau^+$.
Then $\gamma_0 = \cone(e_{n+1},e_{n+2})$ 
is the desired $X$-relevant face.

\medskip
\noindent
\emph{Case 2.}
We have $\deg(S_1) \in \tau^-$ 
and $\tau^+$ contains no degrees of 
variables $S_j$.
If $\deg(T_i) \in \tau^+$ holds 
for some $1 \le i \le q$,
then $\gamma_0 = \cone(e_i,e_{n+1})$ 
is the desired $X$-relevant face.

Suppose that there is no 
$1 \le i \le q$ with 
$\deg(T_i) \in \tau^+$.
Then one and hence all $\deg(T_{q+i})$ 
lie on the ray through $v_1^+$.
This implies $q=0$.
Now, $\cone(e_1,e_2,e_{t+1})$ is 
an $X$-relevant face.
Since $X$ is locally factorial,
the correponding degrees generate
$\Cl(X)$.
We conclude $\Cl(X) = \ZZ^2 \oplus \Gamma$
with a cyclic torsion part $\Gamma = \ZZ/k\ZZ$ 
and, applying a suitable automorphism 
of $\Gamma$, we achieve
$$ 
\deg(T_1) \ = \ (1,0,\overline{0}),
\quad
\deg(T_2) \ = \ (1,0,\overline{1}),
\quad
\deg(S_1) \ = \ (0,1,\overline{0}).
$$
Because of $2\deg(T_1) = 2\deg(T_2)$, we 
obtain $k = 2$. 
Since the $T_i$ must have different degrees 
in $\Cl(X)$, we obtain that there are no 
$T_i$ for $i \ge 3$. 
Thus, no $1 \le i \le q$ with 
$\deg(T_i) \in \tau^+$ is impossible.

\medskip
\noindent
\emph{Case 3.}
The are no $\deg(S_j)$ in 
$\tau^- \cup \tau^+$.
Then we may assume $v_1^- = \deg(T_1)$
and obtain the desired $X$-relevant 
face $\gamma_0 = \cone(e_1,e_i)$
by choosing $i \ne 2$ such that $\deg(T_i)$ 
is one of $v_1^+,v_2^+$.  
\end{proof}

\begin{proof}[Proof of Theorem~\ref{thm:smoothrhoX2}, Part I]
We show that all smooth intrinsic quadrics of 
Picard number $\rho(X)=2$ are isomorphic to 
one of the varieties described in 
Theorem~\ref{thm:smoothrhoX2}.
Proposition~\ref{prop:PicXtorsfree} yields 
$\Cl(X) = \ZZ^2$.
Moreover, according to Proposition~\ref{prop:iq2siq},
we may assume that $X$ is a standard intrinsic 
quadric.
Then the Cox ring of $X$ is given as 
\begin{eqnarray*}
\mathcal{R}(X)
& = & 
\KK[T_1,\ldots,T_n,S_1, \ldots, S_m] / \bangle{g},
\\[1ex]
g  
& = & 
\begin{cases}
T_1T_2 + \ldots + T_{n-1}T_n, & \text{ if } n \text{ is even},
\\
T_1T_2 + \ldots + T_{n-2}T_{n-1} + T_n^2, & \text{ if } n \text{ is odd}.
\end{cases}
\end{eqnarray*}
Note that we have $n \ge 5$, because $\Cl(X)$ 
is torsion free and thus $\mathcal{R}(X)$ 
must be a unique factorization domain.
As outlined in Remark~\ref{rem:effrho2}, 
the effective cone of $X$ is the disjoint 
union of three convex sets,
$$
\Eff(X) \ = \ \tm \cup \tau_X^\circ \cup \tp,
$$
where $\tau_X^\circ \subseteq \Cl_\QQ(X)$ 
is the ample cone.
Since $X$ is smooth, there are no $X$-relevant 
faces of the form $\cone(e_i,e_j)$
with $n+1 \le i < j \le n+m$,
see Proposition~\ref{prop:smooth}.
Consequently the $\deg(S_j)$ 
either all lie in $\tau^-$ or all in $\tau^+$.
After suitably renumbering the variables 
$T_i$ and $S_j$, we are left with the 
following cases:

\medskip

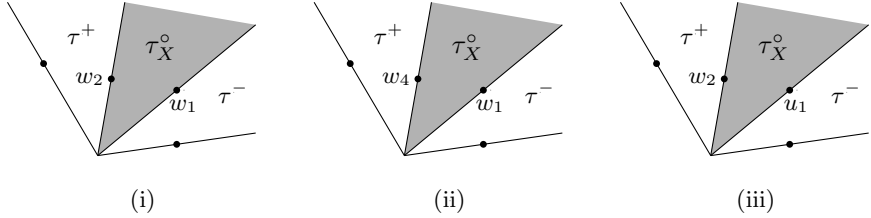
\begin{tikzpicture}[scale=0.6]
    \path[fill=gray!60!] (0,0)--(3.5,2.9)--(0.6,3.4)--(0,0);
    \path[fill, color=black] (1.75,1.45) circle (0.5ex)  node[]{};
    \path[fill, color=black] (1.9,1.45) circle (0ex)  node[below]{\small{$w_1$}};
    \path[fill, color=black] (1.4,2.4) circle (0.0ex)  node[]{$\tau_X^\circ$};
    \path[fill, color=black] (0.3,1.7) circle (0.5ex)  node[]{};
    \path[fill, color=black] (0.35,1.7) circle (0ex)  node[left]{\small{$w_2$}};
    \draw (0,0)--(0.6,3.4);
    \draw (0,0) --(-2,3.4);
    \path[fill, color=black] (-1.2, 2.04) circle (0.5ex)  node[]{};
    \path[fill, color=black] (-0.35,2.7) circle (0.0ex)  node[]{\small{$\tau^+$}};
    \draw (0,0)  -- (3.5,2.9);
    \draw (0,0)  -- (3.5,0.5);
    \path[fill, color=black] (1.75,0.25) circle (0.5ex)  node[]{};
    \path[fill, color=black] (3,1.4) circle (0.0ex)  node[]{\small{$\tau^-$}};
    \path[fill, color=white] (4,1.9) circle (0.0ex);
    \path[fill, color=black] (1,-1) circle (0.0ex)  node[]{\small{(i)}};
\end{tikzpicture}     
\quad
\begin{tikzpicture}[scale=0.6]
    \path[fill=gray!60!] (0,0)--(3.5,2.9)--(0.6,3.4)--(0,0);
    \path[fill, color=black] (1.75,1.45) circle (0.5ex)  node[]{};
    \path[fill, color=black] (1.9,1.45) circle (0ex)  node[below]{\small{$w_1$}};
    \path[fill, color=black] (1.4,2.4) circle (0.0ex)  node[]{$\tau_X^\circ$};
    \path[fill, color=black] (0.3,1.7) circle (0.5ex)  node[]{};
    \path[fill, color=black] (0.35,1.7) circle (0ex)  node[left]{\small{$w_4$}};
    \draw (0,0)--(0.6,3.4);
    \draw (0,0) --(-2,3.4);
    \path[fill, color=black] (-1.2, 2.04) circle (0.5ex)  node[]{};
    \path[fill, color=black] (-0.35,2.7) circle (0.0ex)  node[]{\small{$\tau^+$}};
    \draw (0,0)  -- (3.5,2.9);
    \draw (0,0)  -- (3.5,0.5);
    \path[fill, color=black] (1.75,0.25) circle (0.5ex)  node[]{};
    \path[fill, color=black] (3,1.4) circle (0.0ex)  node[]{\small{$\tau^-$}};
    \path[fill, color=white] (4,1.9) circle (0.0ex);
    \path[fill, color=black] (1,-1) circle (0.0ex)  node[]{\small{(ii)}};
\end{tikzpicture}     
\quad
\begin{tikzpicture}[scale=0.6]
    \path[fill=gray!60!] (0,0)--(3.5,2.9)--(0.6,3.4)--(0,0);
    \path[fill, color=black] (1.75,1.45) circle (0.5ex)  node[]{};
    \path[fill, color=black] (1.9,1.45) circle (0ex)  node[below]{\small{$u_1$}};
    \path[fill, color=black] (1.4,2.4) circle (0.0ex)  node[]{$\tau_X^\circ$};
    \path[fill, color=black] (0.3,1.7) circle (0.5ex)  node[]{};
    \path[fill, color=black] (0.35,1.7) circle (0ex)  node[left]{\small{$w_2$}};
    \draw (0,0)--(0.6,3.4);
    \draw (0,0) --(-2,3.4);
    \path[fill, color=black] (-1.2, 2.04) circle (0.5ex)  node[]{};
    \path[fill, color=black] (-0.35,2.7) circle (0.0ex)  node[]{\small{$\tau^+$}};
    \draw (0,0)  -- (3.5,2.9);
    \draw (0,0)  -- (3.5,0.5);
    \path[fill, color=black] (1.75,0.25) circle (0.5ex)  node[]{};
    \path[fill, color=black] (3,1.4) circle (0.0ex)  node[]{\small{$\tau^-$}};
    \path[fill, color=white] (4,1.9) circle (0.0ex);
    \path[fill, color=black] (1,-1) circle (0.0ex)  node[]{\small{(iii)}};
\end{tikzpicture}   

\noindent
Here, we set $w_i \sei \deg(T_i)$ and $u_j \sei \deg(S_j)$.
Observe that in Case~(iii), we may indeed assume 
$u_1 \in \tau^-$, because 
$\gamma_{2,n+1}$ is an $X$-relevant
face and thus Proposition~\ref{prop:smooth},
allows us to interchange $\tau^-$ and $\tau^+$ 
via a linear coordinate change if necessary.
We now go through the cases, using the notation of
Remark~\ref{rem:Xfaces} for $X$-relevant faces
and writing $\mu = (\mu_1,\mu_2) \in \ZZ^2$
for the degree of $g$.

\medskip

\noindent
\emph{Case~(i):} 
We have $\tau_X = \cone(w_1,w_2)$ with
$w_1 \in \tau^-$ and $w_2 \in \tau^+$.
Then $\mu \in \tau_X$ holds.
Thus, we may assume $w_3 \in \tau^-$
and $w_4 \in \tau^+$.
Applying Proposition~\ref{prop:smooth} 
to~$\gamma_{1,4}$, we see that $w_1,w_4$
form a $\ZZ$-basis for $\ZZ^2$.
By a suitable coordinate change,
we achieve $w_1=(1,0)$ and $w_4=(0,1)$.
Then $w_1+w_2 = w_3+w_4 = \mu$ implies
$w_2 = (\mu_1-1, \mu_2)$
and $w_3 = (\mu_1, \mu_2-1)$.
Like $w_1, w_4$ also $w_3,w_2$ form 
a $\ZZ$-basis for $\Cl(X)$,
being positively oriented, because
$\Eff(X)$ is pointed and we have 
$w_2\in\tau^+$ and $w_3\in\tau^-$.
This implies
$$ 
1 
\ = \ 
\det(w_3,w_2)
\ = \ 
\mu_1 + \mu_2 - 1.
$$
From $\mu \in \tau_X \subseteq \cone(w_1,w_4)$
we infer $\mu_1,\mu_2 > 0$ and conclude
$\mu_1 = \mu_2 = 1$.
In particular, we have 
$w_2=(0,1),\,w_3=(1,0)$
and
$\tau_X=\QQ^2_{\ge0}$. 
Moreover, $\mu=(1,1)$ implies that 
$n$ is even.
Suitably renumbering the $T_i$ with
$i \ge 5$,
we achieve $w_i \in \tau^-$ and 
$w_{i+1} \in \tau^+$ 
for $i = 5, 7, \ldots, n-1$.
Then, for every odd $i$, 
Proposition~\ref{prop:smooth} 
and homogeneity of $g$
provide us with the conditions
$$ 
\det(w_i,w_2) = 1,
\qquad
w_i + w_{i+1} = \mu = (1,1),
\qquad
\det(w_1,w_{i+1}) = 1.
$$
We conclude $w_i = (1,0)$ and $w_{i+1} = (0,1)$
for all $i = 5, 7, \ldots, n-1$.
The weights $u_j = \deg(S_j)$ are contained 
either all in $\tau^-$ or all in $\tau^+$.
We may assume all in $\tau^+$. 
Applying Proposition~\ref{prop:smooth} to
$\gamma_{1,j}$, where $j = 1, \ldots, m$ 
yields $u_j=(a_j,1)$ with some 
$a_j \in \ZZ_{\le 0}$.
A suitable linear coordinate change
in $\ZZ^2$ leads to Type~4.

\medskip
\noindent
\emph{Case~(ii):} 
We have $\tau_X = \cone(w_1,w_4)$ with
$w_1 \in \tau^-$ and $w_4 \in \tau^+$.
If $w_2 \in \tau^+$ holds, then 
we are in Case~(i) just settled.
We treat the case $w_2 \in \tau^-$.
Since $\gamma_{1,4}$ is an $X$-relevant 
face, Proposition~\ref{prop:smooth} 
says that $w_1,w_4$ form a $\ZZ$-basis 
of $\ZZ^2$.
Thus, a suitable linear coordinate 
change in $\ZZ^2$ yields $w_1=(1,0)$ 
and~$w_4=(0,1)$.

As $\gamma_{2,4}$ is an $X$-relevant face,
we can apply Proposition~\ref{prop:smooth}
again and see that $w_2,w_4$ is a positively 
oriented $\ZZ$-basis of $\ZZ^2$. 
Thus, $\det(w_2,w_4) = 1$ holds and we 
conclude $w_2 = (1,x)$ with some $x \in \ZZ_{\le 0}$.
By the same arguments, if $v^+ \in \tau^+$ 
is a degree of any of the $T_i \ne T_4$ or 
the $S_j$,
we see that $w_1,v^+$ is a positively oriented
$\ZZ$-basis and conclude $v^+ = (y,1)$
with some $y \in \ZZ_{\le 0}$. 
Moreover, arguing further along this line gives 
$$ 
1 = \det(w_2,v^+) = 1 -xy,
\qquad\qquad
1 
= 
\det(w_3,v^+) 
=
2 - xy + y,
$$
where we use $w_3 = \mu - w_4 = w_1+w_2-w_4$ 
for the last equality. 
This implies $x=0$ and $y=-1$ and thus 
$w_2 = (1,0)$ and $v^+ = (-1,1)$.
We conclude $\mu = w_1+w_2 = (2,0)$ 
and $w_3 = \mu-w_4 = (2,-1)$.
So far, the situation looks as follows:
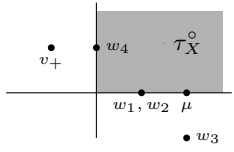
\begin{center}
\begin{tikzpicture}[scale=0.6]
\path[fill=gray!60!] (0,0)--(0,1.8)--(2.8,1.8)--(2.8,0)--(0,0);
\path[fill, color=black] (1.5,1.1) circle (0.0ex)  node[right]{\small $\tau_X^\circ$};
\draw (-2,0)--(3,0);
\draw (0,-1)--(0,2);
\path[fill, color=black] (1,0) circle (0.5ex)  node[below]{\tiny $w_1,w_2$};
\path[fill, color=black] (2,0) circle (0.5ex)  node[below]{\tiny $\mu$};
\path[fill, color=black] (0,1) circle (0.5ex)  node[right]{\tiny $w_4$};
\path[fill, color=black] (-1,1) circle (0.5ex)  node[below]{\tiny $v_+$};
\path[fill, color=black] (2,-1) circle (0.5ex)  node[right]{\tiny $w_3$};
%
\end{tikzpicture}   
\end{center}
In particular, if $n$ is odd, we must have 
$w_n = (1,0)$. 
Because of $\mu \in \tau^-$, we can 
renumber the other $T_i$ with $i \ge 5$
such that $w_i \in \tau^-$ holds for all odd $i$.
Now, consider any odd $i$ with $5 \le i < n$.
Because of $\mu = (2,0)$ we either 
have $w_i = w_{i+1} = (1,0)$ 
or $w_{i+1} =  v^+ = (-1,1)$ and $w_i = (3,-1)$.
The second case is excluded, because then 
$\gamma_{4,i}$ is an $X$-relevant face, 
contradicting the smoothness condition
of Proposition~\ref{prop:smooth}.
Thus, $w_i = (1,0)$ holds for all $i \ge 5$.
Consequently, there must be at least one 
$u_j$ and all $u_j$ equal $v^+ = (-1,1)$.
A suitable linear coordinate change in $\ZZ^2$ 
and renumbering the variables leads to Type~3.

\medskip

\noindent
\emph{Case~(iii):} 
We have $\tau_X = \cone(u_1,w_2)$
with $u_1\in\tau^-$ and $w_2\in\tau^+$.
Then $\gamma_{2,n+1}$ is an $X$-relevant 
face.
By Proposition~\ref{prop:smooth}, we 
achieve $u_1=(1,0)$ and $w_2=(0,1)$
via a suitable linear coordinate change.
We distinguish the following two subcases.

First assume $w_1 \in \tau^+$.
Then $\mu \in \tau^+$ holds.
Thus, we may assume that all $w_i$ 
with $i$ odd are contained
in~$\tau^+$. 
Proposition~\ref{prop:smooth} shows that 
$u_1,w_i$ is a $\ZZ$-basis for the odd 
$i < n$ and thus $w_i = (x_i,1)$ 
holds in these cases, 
where $x_i \in \ZZ_{\le 0}$ due to 
$w_i \in \tau^+$.  
In particular, we have
$$
\mu = w_1 + w_2=(x_1, 2),
\qquad
w_{i+1} = \mu - w_i \ = \ (x_1-x_i,1),
$$ 
where $i < n$ is odd.
Thus, we obtain $w_1, w_2, \ldots, w_n \in \tau^+$.
Consequently, $m \ge 2$ holds.
Because of $u_1 \in \tau^-$, we have 
$u_j \in \tau^-$ for all $j = 1, \ldots, m$.
Moreover $u_j,w_2$ is a $\ZZ$-basis due to
Proposition~\ref{prop:smooth} and thus 
$u_j = (1,y_j)$ holds, where $y_j \in \ZZ_{\le 0}$.
Repeating the same argument with all pairings
$u_j,w_i$ yields $w_i = (0,1)$ for all $i$ 
or $u_j = (1,0)$ for all $j$.
Applying a suitable linear coordinate change, 
we arrive at Type~1 or Type~2, respecitvely.

Now assume $w_1 \in \tau^-$.
Then $\mu \in \tau_X \cup \tau^-$ holds.
Suitably renumbering the $T_i$, we achieve 
$w_i \in \tau^-$ for all odd $i < n$.
Moreover, as all $u_j$ lie in $\tau^-$
and there must be the degree of a second
variable in $\tau^+$, we may assume 
$w_4 \in \tau^+$.
Proposition~\ref{prop:smooth} applied 
to the $X$-relevant faces 
$\gamma_{2,i}$ for the $i \ge 3$ with 
$w_i \in \tau^-$ 
and 
$\gamma_{i,n+1}$ for the $i$ with 
$w_i \in \tau^+$
shows
$$ 
w_i 
\ = \ 
\begin{cases}
(1,y_i) \text{ with }  y_i \in \ZZ_{\le 0} & \text{if } w_i \in \tau^-,
\\
(x_i,1) \text{ with }  x_i \in \ZZ_{\le 0} & \text{if } w_i \in \tau^+.
\end{cases}
$$
unless $i=1$ or $i=n$ with $n$ odd.
Now consider any even $i$ with 
$4 \le i \le n$.
Then the degree of $g$ is 
$\mu = w_3+w_4 = (1+x_4,y_3+1)$.
Because of $\mu \in \tau_X \cup \tau^-$, 
we conclude $x_4 = 0$ and obtain
$\mu = (1,y_3+1)$. 
In particular, we see that $n$ is even and 
$w_i \in \tau^+$ holds for all even $i$.
Moreover,
$w_1 = \mu - w_2  = w_3$ holds
and for every even~$i$ we have
$$ 
(1,y_3+1) 
\ = \ 
\mu
\ =  \
w_{i-1}+w_i
\ = \ 
(1+x_i,y_{i-1}+1).
$$
Consequently, $x_i=0$ and thus $w_i=(0,1)$ 
holds for all even $i$.
Thus, for the odd~$i$, we obtain 
$w_i = (1,y_3)$.
Finally $u_j,w_2$ is a $\ZZ$-basis
for all $j = 1, \ldots, m$ and thus
we have $u_j = (1,a_j)$ with $a_j \in \ZZ_{\le 0}$.
So, a suitable linear coordinate change 
leads to Type~4.
\end{proof}

\section{Picard number two:~geometry}

We discuss geometric aspects of the 
intrinsic quadrics listed in Theorem~\ref{thm:smoothrhoX2}.
First, we enter their Mori theory and
prove the still open geometric statements 
made in Theorem~\ref{thm:smoothrhoX2}.
Then we figure out the Fano examples 
from Theorem~\ref{thm:smoothrhoX2}
and thus prove Corollary~\ref{cor:FanosAlmFanos}.
Moreover, we obtain base point freeness for
numerically effective divisors,
see Corollary~\ref{cor:bpfsat},
and thus can verify Fujita's freeness
conjecture for all smooth intrinsic quadrics 
of Picard number at most two,
see Corollary~\ref{cor:bpfquadricssatrhoX2}.
Finally, we discuss Mukai's conjecture in 
Example~\ref{cor:Mukai}. 

The morphisms providing the geometric 
descriptions of Theorem~\ref{thm:smoothrhoX2}
are examples of so called elementary 
contractions~\cite{Ca}.
We obtain them by looking at the 
\emph{Mori chamber decomposition}, 
which in our case is easy to compute. 
Before entering the details, let us 
briefly recall some general background
background.
Every effective divisor $D$ on a normal projective 
variety~$X$ defines a rational map
$$ 
\varphi_D \colon
X \ \dasharrow \ X(D),
\qquad\qquad
X(D) 
\ \sei \ 
\Proj
\Bigl( 
\bigoplus_{n \in \ZZ_{\ge 0}} \Gamma(X,\mathcal{O}_X(nD))
\Bigr).
$$  

Two divisors are called \emph{Mori equivalent} if they define 
the same map.
The Mori chamber decomposition is the 
subdivision of the effective cone into the classes arising 
from Mori equivalence.
In the case of a Mori dream space $X$,
there is a fundamental connection to 
geometric invariant theory, as observed 
by Hu and Keel~\cite{HuKe}.
Namely, we have the action of 
the quasitorus 
$H = \Spec \, \KK[\Cl(X)]$ 
on $\overline{X} \sei \Spec \, \mathcal{R}(X)$
and thus the \emph{GIT-fan} 
$\Lambda_X$ describing the variation 
of GIT-quotients in the sense that two 
classes $w_1,w_2 \in \Cl(X)$ 
define the same sets of semistable points 
$\overline{X}^{ss}(w_i)$ 
if and only if they lie in the relative 
interior of a common cone $\lambda \in \Lambda_X$.
Now, the crucial observation is
that, inside the moving cone,
the Mori chambers of $X$  
are the precisely the relative 
interiors of the cones of the GIT-fan.

For a Mori dream space 
$X = \overline{X}^{ss}(u) \quot H$,
the cone $\lambda(u) \in \Lambda_X$ 
containing $u$ in its relative interior, 
is the semiample cone of $X$.
A divisor $D$ defines a morphism 
$\varphi_D \colon X \to X(D)$
if and only if for the class $w$ of $D$,
the associated cone $\lambda(w) \in \Lambda_X$ 
is a face of $\lambda(u)$.  
In this case, $\varphi_D \colon X \to X(D)$
is called a \emph{contraction} and, in the 
GIT picture, $\varphi_D$ is the induced 
map of GIT-quotients making the following 
diagram commutative:
$$ 
\xymatrix{
{\overline{X}}^{ss}(u)
\ar@{}[r]|\subseteq
\ar[d]_{\quot H}
&
{\overline{X}}^{ss}(w)
\ar[d]^{\quot H}
\\
X
\ar[r]_{\varphi_D}
&
X(D)
}
$$
A contraction $\varphi \colon X \to X(D)$ 
is called \emph{elementary} if $X(D)$ is 
of Picard number one less than~$X$.
There are three possibilities for such an 
elementary contraction, according to the 
possible positions of the class of $D$ in the 
effective cone: 
\begin{itemize}
\item
The class of $D$ lies on the boundary of $\Eff(X)$.
Then $\varphi_D$ is of \emph{fiber type}, i.e., 
the dimension of $X(D)$ is strictly less than 
that of $X$.
\item
The class of $D$ lies on a boundary 
of $\Mov(X)$,
but not on the boundary of $\Eff(X)$. 
Then $\varphi_D$ is a \emph{birational divisorial contraction}, 
i.e., it is birational and contracts precisely a
prime divisor of $X$.
\item
The class of $D$ lies in the interior $\Mov(X)$.
Then $\varphi_D$ is a \emph{birational small 
contraction}, i.e., it is birational and contracts 
only a subvariety of codimension at least two.
\end{itemize}

\goodbreak

\begin{remark}
Construction~\ref{constr:intquad}
produces an intrinsic quadric $X$ 
in an ambient toric variety $Z$ 
by passing to a quotient
of  the action of $H$ on $\overline{X}$ 
and $\overline{Z} = \KK^{n+m}$.
The cones of the (finite) GIT-fans 
$\Lambda_X$ and $\Lambda_Z$ 
in $K_\QQ = \Cl_\QQ(X) = \Cl_\QQ(Z)$ are  
$$ 
\lambda_X(w) 
\ = \ 
\bigcap_{
\tiny
\begin{array}{c}
w \in Q(\gamma_0), 
\\
\overline{X}(\gamma_0) \ne \emptyset
\end{array}
}
Q(\gamma_0),
\qquad\qquad
\lambda_Z(w) 
\ = \ 
\bigcap_{w \in Q(\gamma_0)}
Q(\gamma_0),
$$
respectively, 
where $w$ runs through $K_\QQ$ and 
Remark~\ref{rem:Xfaces} tells which are
the faces $\gamma_0 \preceq \gamma = \QQ^{n+m}$ 
such that $\overline{X}(\gamma_0)$ is non-empty.
In particular, the fan~$\Lambda_Z$ 
refines the fan $\Lambda_X$,
which in turn connects the Mori theory of $X$ 
with that of $Z$.
\end{remark}

\begin{proof}[Proof of Theorem~\ref{thm:smoothrhoX2}, Part~II]
We first discuss the varieties $X$ of Types~1, 2 and~4. 
In these cases, the configurations of weights and the 
semiample cone are of the following shape:
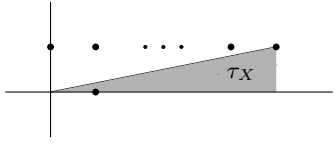
\begin{center}
\begin{tikzpicture}[scale=0.6]

\draw[thin] (0,0)--(5,1);
\path[fill=gray!60!] (0,0)--(5,0)--(5,1)--(0,0);
\path[fill, color=black] (3.7,0.4) circle (0.0ex)  node[right]{\small $\tau_X$};

\draw (-1,0)--(6.25,0);
\draw (0,-1)--(0,2);

\path[fill, color=black] (1,0) circle (0.5ex);
\path[fill, color=black] (0,1) circle (0.5ex);
\path[fill, color=black] (1,1) circle (0.5ex);

\path[fill, color=black] (2.1,1) circle (0.3ex);
\path[fill, color=black] (2.5,1) circle (0.3ex);
\path[fill, color=black] (2.9,1) circle (0.3ex);

\path[fill, color=black] (4,1) circle (0.5ex);
\path[fill, color=black] (5,1) circle (0.5ex);
\path[fill, color=black] (5,.6) circle (0ex);
\end{tikzpicture}   
\end{center}

\noindent
We work with the toric embedding $X \subseteq Z$ 
provided by Construction~\ref{constr:intquad}.
From Remark~\ref{rem:Xfaces}, we infer
$\tau_X = \tau_Z$ for the semiample 
cones. 
Thus, for the divisor class $w = (1,0)$,
a representing toric divisor $E$ on $Z$ 
and its restriction $D$ on $X$, we obtain 
a commutative diagram 
$$ 
\xymatrix{
X
\ar@{}[r]|\subseteq
\ar[d]_{\varphi_D}
&
Z
\ar@{}[r]|{\cong\qquad\qquad}
\ar[d]^{\varphi_E}
& 
{\PP(\mathcal{O}(a_1) \oplus \ldots \oplus \mathcal{O}(a_k))}
\ar[d]
\\
X(D)
\ar@{}[r]|\subseteq
&
Z(E)
\ar@{}[r]|{\cong\qquad\qquad}
& 
{\PP_{l-1}}
}
$$
where the inclusions are closed embeddings,
$l$ is the number of coordinates of 
$\overline{Z} = \KK^{n+m}$ of degree
$(1,0)$, call them $f_1, \ldots, f_l$, 
and $k$ is the number of remaining 
coordinates, 
call them $h_1, \ldots, h_k$.
So, we have $n+m = l+k$.
In terms of homogeneous coordinates
on $\overline{Z}$ and on $\overline{Z(E)} = \KK^{l}$,
local trivializations of the bundle projection
$\varphi_E \colon Z \to Z(E)$ are given by
$$ 
\xymatrix{
{\KK^{n+m}_{f_i} \setminus V(h_1, \ldots, h_k)}
\ar[rrrr]^{\left(f_1,\ldots,f_l,\frac{h_1}{f_i^{a_1}}, \ldots,\frac{h_k}{f_i^{a_k}}\right)}
\ar[drr]_{(f_1,\ldots,f_l)\quad}
&&&&
{\KK^{l}_{f_i} \times (\KK^{k} \setminus \{0\})}
\ar[dll]^{\pr_{\KK^{l}}}
\\
&& 
{\KK^{l}_{f_i}}
&& 
}
$$

If $X$ is of Type~1, then $l=n$ 
and $k=m$ hold,
the $f_i$ are the variables $T_i$ and 
the~$h_j$ are the variables $S_j$.
We directly see that $X$ maps onto  
$V(g) \subseteq \PP_{n-1}$ and that
$X = \varphi_E^{-1}(V(g))$ holds.
Thus, $\varphi_D \colon X \to X(D)$ 
is a bundle projection as wanted.

If $X$ is of Type~2, then $l=m$ 
and $k=n$ hold,
the $f_i$ are the variables $S_i$ and 
the~$h_j$ are the variables $T_j$.
Using the fact that the relation $g$ 
is $K$-homogeneous, we see that the above 
local trivializations respect $g$.
We conclude that $X(D) = Z(E) = \PP_{m-1}$ 
holds and that locally with respect to the 
base~$X(D)$, the variety $X$ is a product 
of the smooth quadric $V(g) \subseteq \PP_{n-1}$ 
and $\PP_{m-1}$.

If $X$ is of Type~4, then $l=n/2$ 
and $k=m+n/2$ hold,
the $f_i$ are the variables~$T_i$ with 
$i$ odd and the $h_j$ subsume the 
variables $T_i$ with $i$ even as well 
as the variables~$S_j$.
Using the above local trivializations, 
we see that $X$ projects onto the base, that means that 
$X(D) = Z(E) = \PP_{n/2-1}$ holds.
Morover, on each fiber $\varphi_E^{-1}([z])$, 
the relation~$g$ becomes a linear form 
in the coordinates $T_i$ with $i$ even 
and thus cuts out a hyperplane of 
$\varphi_E^{-1}([z]) \cong \PP_{n/2+m-1}$.
Consequently, $\varphi_D \colon X \to X(D)$ 
is as claimed.

Finally, let $X$ be of Type~3.
Observe, that the ambient toric variety $Z$ 
is not smooth in this case.
We take the divisors $E$ on $Z$ and 
$D$ on $X$ corresponding to the 
generator~$T_3$ of the respective Cox 
rings.
Then $\varphi_E \colon Z \to Z(E)$ and 
$X \to X(E)$ contract the respective 
divisors defined by~$T_1$.
We obtain that $Z(E)$ is the weighted
projective space $\PP_{(1,\ldots,1,2)}$
of dimension $n+m-2$
and $X(D) \subseteq Z(E)$ is defined 
by the equation 
$g-T_1T_2 + T_2 = 0$.
This gives $X(D) = \PP_{n+m-3}$.
Moreover, the fan of the ambient toric 
variety $Z$ is obtained from the fan of  
$Z(E) = \PP_{(1,\ldots,1,2)}$ by barycentric
subdivsion of the cone over the rays 
corresponding to $u_1,\ldots,u_m,w_2$.
We conclude that the center of 
the modification $X \to X(D)$ is 
the smooth quadric 
$V(g-T_1T_2,S_1,\ldots,S_m) \subseteq \PP_{n+m-3}$.
The fact that $X \to X(D)$ is indeed the 
blowing-up is checked directly in the
affine charts of $X(D) = \PP_{n+m-3}$.
\end{proof}

We turn to the (almost) Fano varieties 
among the smooth intrinsic 
quadrics of Picard number two.
Adapting~\cite{ArDeHaLa}*{Prop.~3.3.3.2}
leads to the following 
explicit desciption of the anticanonical
class.

\begin{proposition}
\label{prop:antican}
Let $X = X(q,t,m,u)$ be a standard intrinsic quadric
arising from Construction~\ref{prop:quaddiag}.
Then the anticanonical class of $X$ is given by
$$ 
- \mathcal{K}_X
\ = \ 
\frac{q-2}{2} \deg(g_{q,t}) 
+ 
\sum_{i=1}^{t} \deg(T_{q+i})
+ 
\sum_{k=1}^{m} \deg(S_k)
\ \in \ 
K
\ = \ 
\Cl(X). 
$$
\end{proposition}

\begin{proof}[Proof of Corollary~\ref{cor:FanosAlmFanos}]
In the situation of Theorem~\ref{thm:smoothrhoX2}, 
the formula of Proposition~\ref{prop:antican} simplifies 
to
$$
-\mathcal{K}_X
\ = \ 
\frac{n-2}{2} \deg(g) + \deg(u_1) + \ldots + \deg(u_m).
$$
The variety $X$ is Fano if and only if $-\mathcal{K}_X$ lies
in the interior of the semiample cone $\tau_X$ specified 
in the Theorem and $X$ is truly almost Fano if $-\mathcal{K}_X$ 
lies on the boundary of $\tau_X$
and in the interior of $\Eff(X)$.
One immediately computes:

\begin{center}
\renewcommand{\arraystretch}{1.8} 
\begin{tabular}{c|c|c}
Type
& 
$-\mathcal{K}_X$
& 
$\tau_X$
\\
\hline
1
&
$
\left({n-2 \atop 0}\right)
+
\left({a_1 + \ldots +a_m \atop m}\right)
$
&
$
\cone 
\left(
\left({1 \atop 0}\right),
\left({\alpha \atop 1}\right)
\right)
$
\\
\hline
2
&
$
\frac{n-2}{2}\left({\alpha \atop 2}\right)
+
\left({m \atop 0}\right)
$
&
$
\cone 
\left(
\left({1 \atop 0}\right),
\left({\alpha \atop 1}\right)
\right)
$
\\
\hline
3
&
$
{\scriptstyle (n-2)}
\left({1 \atop 1}\right)
+
\left({m \atop 0}\right)
$
&
$
\cone 
\left(
\left({1 \atop 1}\right),
\left({2 \atop 1}\right)
\right)
$
\\
\hline
4
&
$
\frac{n-2}{2}\left({a+1 \atop 1}\right)
+
\left({a_1 + \ldots + a_m \atop m}\right)
$
&
$
\cone 
\left(
\left({1 \atop 0}\right),
\left({\alpha \atop 1}\right)
\right)
$
\end{tabular}
\end{center}
From this, we directly derive the Fano and truly almost Fano 
conditions.
Note that for Type~4, we must have $w_2 = (\alpha,1)$ 
in order to obtain a (truly almost) Fano variety
and $w_2 = (0,1)$ produces further truly almost Fano 
varieties.
\end{proof}

\begin{corollary}
\label{cor:bpfsat}
Let $X$ be a smooth intrinsic quadric of 
Picard number at most two.
Then every numerically effective 
divisor on $X$ is base point free.
\end{corollary}

\begin{proof}
We may assume that $X$ arises from 
Construction~\ref{prop:quaddiag}.
Consider the monoid
$\BPF(X) \subseteq \Cl(X)$ 
of divisor classes admitting a 
base point free representative.
Using~\cite{ArDeHaLa}*{Prop.~3.3.2.8},
we obtain
$$ 
\BPF(X) 
\ = \ 
\bigcap_{\gamma_0 \in \cov(X)} Q(\gamma_0 \cap \ZZ^{n+m})
\ \subseteq \
\Cl(X).
$$
By Proposition~\ref{prop:ample}, the 
cone in $\Cl_\QQ(X)$ generated by 
$\BPF(X)$ equals the cone $\SAmple(X)$
of semiample divisor classes.
As for any Mori dream space, 
$\SAmple(X)$ coincides with the cone 
of numerically effective divisor classes.

Our task is to show that $\BPF(X)$ is 
saturated in $\Cl(X)$ in the sense
that given $w \in \Cl(X)$ and 
$n \in \ZZ_{>0}$ with $nw \in \BPF(X)$,
one has $w \in \BPF(X)$.
Since the intersection of saturated 
submonoids is saturated, 
it suffices to show that every 
monoid
$Q(\gamma_0 \cap \ZZ^{n+m})$,
where $\gamma_0 \in \cov(X)$,
is saturated.
If $X$ is of Picard number one,
then Proposition~\ref{prop:picardnumber1}
tells us $\Cl(X) = \ZZ$ 
and $Q(\gamma_0 \cap \ZZ^{n+m}) = \ZZ_{\ge 0}$ 
for all faces $\gamma_0 \in \cov(X)$,
proving that $\BPF(X)$ is saturated.

Assume that $X$ is of Picard number two.
Then $\Cl(X) = \ZZ^2$ holds according to
Theorem~\ref{thm:smoothrhoX2}.
Moreover, for any two-dimensional face
$\gamma_{i,j} = \cone(e_i,e_j)$ of $\cov(X)$,
Proposition~\ref{prop:smooth}~(v) says 
that $Q(e_i)$ and $Q(e_j)$ form a $\ZZ$-basis 
for $\Cl(X)$.
We conclude that $Q(\gamma_{i,j} \cap \ZZ^{n+m})$
is saturated for all two-dimensional
faces $\gamma_{i,j} = \cone(e_i,e_j)$ of~$\cov(X)$.
Theorem~\ref{thm:smoothrhoX2} specifies 
the semiample cone for each of the Types~1 to~4.
Combining this with Remark~\ref{rem:Xfaces}
allows us to determine the set $\cov(X)$ 
of minimal $X$-relevant faces explicitly.
If $X$ is of Types~1, 2 or~4, then we 
see that in fact all $\gamma_0 \in \cov(X)$ 
are two-dimensional and thus $\BPF(X)$ is saturated.
We are left with discussing $X$ of Type~3.
If $n$ is even, then again all
$\gamma_0 \in \cov(X)$ are two-dimensional.
If $n$ is odd, then, besides the two-dimensional 
ones, we find one more face in $\cov(X)$, 
namely $\gamma_{1,2,n} = \cone(e_1,e_2,e_n)$.
The corresponding images under $Q$ 
are~$(0,1)$, $(2,1)$ and~$(1,1)$, 
generating the saturated monoid 
$\cone(Q(\gamma_{1,2,n})) \cap \ZZ^2$.
\end{proof}

\begin{corollary}
\label{cor:bpfquadricssatrhoX2}
Every smooth intrinsic quadric $X$ 
of Picard number at most two
fulfills Fujita's freeness
conjecture.
That means that 
$C_X + s D$ is base point free
for any canonical divisor $C_X$,
all~$s \geq \dim(X)+1$ and all ample
divisors $D$ on~$X$.
\end{corollary}

\begin{proof}
Fujita proved that
$C_X + s D$ is numerically effective 
under the above assumptions~\cite{Fu}*{Thm.~1}
on $s$ and $D$.
Thus, Corollary~\ref{cor:bpfsat} 
gives the assertion.
\end{proof}

Mukai's conjecture~\cite{Mukai}
predicts
$\rho(X)(q(X) - 1) \le \dim(X)$
with equality if and only if~$X$ is 
the~$\rho(X)$-fold product of the 
projective space~$\PP_{q(X)-1}$
for every smooth Fano variety $X$
of Picard number $\rho(X)$ and 
Fano index $q(X)$.
The conjecture is proven for 
toric $X$ and in general for
$\rho(X) \le 2$ as well as for 
$\dim(X) \le 5$;
see~\cites{Cas, bocadedr, ACO, Wis}.
Let us revisit the case $\rho(X) \le 2$.

\begin{example}
\label{cor:Mukai}
We show how to obtain Mukai's conjecture 
for smooth Fano intrinsic quadrics~$X$ 
of Picard number $\rho(X) \le 2$
from our results.
In the case $\rho(X) = 1$, 
Proposition~\ref{prop:picardnumber1}
tells us that $X$ is a smooth quadric
in a projective space and thus satisfies 
Mukai's conjecture.
So, assume $\rho(X) = 2$.
We may assume that $X$ arises from
Construction~\ref{prop:quaddiag}
with input data given by 
Theorem~\ref{thm:smoothrhoX2}.
Note that we have 
$$
\dim(X) \ = \ n+m-3.
$$
Corollary~\ref{cor:FanosAlmFanos}
provides us with the Fano condition.
Moreover, the anticanonical class 
$-\mathcal{K}_X \in \Cl(X) = \ZZ^2$
is specified in the table 
shown in the proof of 
Corollary~\ref{cor:FanosAlmFanos}
and the Fano index $q(X)$ 
equals the greatest common divisor 
of the two entries of the vector
$-\mathcal{K}_X$.
We now go through the four different 
Types of Theorem~\ref{thm:smoothrhoX2}.

Let $X$ be of Type~1.
If~$\alpha=0$ holds, then we have 
$-\mathcal{K}_X=(n-2,m)$ 
and thus~$q(X) = \gcd(n-2,m)$. 
We conclude
$$
2(q(X) -1) 
\ \le \
2\min(n-2,m) -2 
\ \le \
(n-2+m)-2
\ < \ 
\dim(X).
$$ 
Now let~$\alpha>0$. 
With $k \sei n-2 + a_1 + \ldots + a_m$,
we have~$-\mathcal{K}_X=(k,m)$ 
and~$q(X)$ divides both entries.
This implies $q(X) \le m$.
If $q(X)<m$ holds, then, because of 
$n \ge 5$, we obtain
$$
2(q(X) -1) 
\ \le \
2\left(\frac{m}{2} - 1 \right) 
\ = \ 
m-2 
\ \le \ 
n+m -7
\ < \
\dim(X).
$$
If we have $q(X)=m$, then~$m$ divides $k$.
Thus, the Fano condition $\alpha m < k$ 
implies $(\alpha+1)m \le k$.
Moreover, $\alpha > 0$ implies 
$a_1 + \ldots + a_m < \alpha m$.
Together, we obtain
$$ 
2(q(X)-1)
\ = \ 
2m-2
\ \le \ 
k - (\alpha + 1-2)m-2
\ < \ 
n-2 + m -2 
\ < \ 
\dim(X).
$$

Let $X$ be of Type~2.
Then we have 
$-\mathcal{K}_X =( (n/2-1) \alpha + m, n-2 )$
and $q(X)$ divides both entries. 
If $q(X)<n-2$ holds, 
then using $m \ge 2$, we obtain 
$$
2(q(X)-1) 
\ \le \
2 \left( \frac{n-2}{2} -1 \right)
\ = \
n-4
\ \le \
n+m-6 
\ < \ 
\dim(X).
$$
We are left with discussing the case $q(X) = n-2$.
If~$\alpha=0$ holds, then we 
obtain $-\mathcal{K}_X=(m, n-2)$ and
thus $n-2 \le m$. We conclude
$$
2(q(X)-1) 
\ = \
2((n-2) -1)
\ = \
(n-2) + (n-4)
\ \le \
n +m -4
\ < \
\dim(X).
$$
Next, let $\alpha = 1$.
Then $-\mathcal{K}_X=((n-2)/2 +m, n-2)$ holds
and $q(X) = n-2$ divides the first entry.
Thus, with a suitable $k\in\ZZ$, we have 
$$
m \ = \ \frac{2k+1}{2}(n-2).
$$
In the case $\alpha = 1$, the Fano condition 
reads as $m > (n-2)/2$ and thus~$k \ge 1$ 
holds.
Moreover, because of $n > 4$, we have 
$n/2 -1 < n-3$ and thus obtain
$$
2(q(X)-1)
\ = \
2(n-2-1)
\ < \
\frac{3}{2}(n-2) + \frac{n-2}{2}
\ \le \
m+\frac{n}{2}-1
\ < \ 
\dim(X).
$$
Finally, let $\alpha \ge 2$. 
Then the Fano condition 
says $n-2 < 2m/\alpha$.
Consequently, we obtain
$$
2(q(X)-1) 
\ = \
(n-2) + (n-4)
\ < \
\frac{2m}{\alpha} + (n-4)
\ \le \
n + m - 4 
\ < \ 
\dim(X).
$$

Let $X$ be of Type~3.
Then~$-\mathcal{K}_X=(n-2+m, n-2)$ holds
and the Fano condition yields~$m < n-2$.
As $q(X)$ divides both entries of $-\mathcal{K}_X$,
we see $q(X) \ne n-2$ and thus
$q(X) \le (n-2)/2$.
We conclude
$$
2(q(X)-1) 
\ \le \
2 \left( \frac{n-2}{2} -1 \right)
\ = \
n-4
\ < \
\dim(X).
$$

Let $X$ be of Type~4.
Then
$-\mathcal{K}_X=((n/2-1) (\alpha+1) + a_1 + \ldots + a_m, n/2-1+m)$
holds.
In the case $\alpha = 0$, all the $a_i$ vanish 
as well, we obtain $q(X)\le n/2-1$ and thus
$$
2(q(X)-1) 
\  \le \
n-4
\ \le \
n+m-4
\ < \
 \dim(X).
$$
Let~$\alpha>0$.
If~$q(X) < n/2-1+m$ holds,
then even $q(X) \le (n/2-1+m)/2$
must hold and, because
of $n < 0$, we obtain 
$$
2(q(X)-1) 
\  \le \
\frac{n}{2} - 1 + m - 2
\ < \
 \dim(X).
$$
We discuss the case $q(X)=n/2-1+m$.
The first component of $-\mathcal{K}_X$
equals $\beta q(X)$ with some positive 
integer $\beta$.
Plugging the Fano condition 
$$
\alpha m - (n/2 -1) \ <  \ a_1 + \ldots + a_m
$$
into this equality leads to the 
estimate $\alpha + 1 \le \beta$. 
Comparing $(\alpha+1)q(X)$ with 
the first component $\beta q(X)$ 
of $-\mathcal{K}_X$ gives
$$
(\alpha+1)q(X) 
\ \le \ 
\left(\frac{n}{2}-1\right)(\alpha+1) + a_1 + \ldots + a_m 
\  <  \ 
\left(\frac{n}{2}-1\right)(\alpha+1) + \alpha m,
$$
where the last inequality is due to the 
fact that $\alpha>0$ forces vanishing 
of at least one of the~$a_j$.
This allows us to conclude the discussion by
\begin{eqnarray*}
2(q(X)-1)
& = &
\big( (\alpha+1) -(\alpha+1-2) \big) q(X) -2
\\
& <  & 
\left(\left(\frac{n}{2}-1\right)(\alpha+1)+\alpha m\right)
-(\alpha-1)  q(X)
-2\\
& = &
n+m-4\\
& <& 
\dim(X).
\end{eqnarray*}
\end{example}

\section{Proof of Theorem~\ref{thm:fullFanos}}
\label{sec:fullFanos}

A first step is the general bound for 
the Picard number of (possibly singular) 
Fano full intrinsic quadrics
provided in Proposition~\ref{prop:fullFano}.
Then we prepare the 
proof of Theorem~\ref{thm:fullFanos},
which is given at the end of the section.
We will mostly work in the setting 
of standard intrinsic quadrics
$X = X(q,t,m,u)$ arising from 
Construction~\ref{constr:intquad}.
We write $g = g_{q,t}$ for the 
relation and the degrees of the 
variables in $\Cl(X) = K$ will be 
denoted as
$$ 
w_i = \deg(T_i) = Q(e_i)
\text{ for }
i = 1, \ldots, n,
$$
$$
w_{n+j} = \deg(S_j) = Q(e_{n+j})
\text{ for }
j = 1, \ldots, m.
$$

\begin{proposition}
\label{prop:fullFano}
Let $X$ be a Fano full standard intrinsic quadric 
arising from Construction~\ref{constr:intquad}.
\begin{enumerate}
\item
If $t > 1$ holds, i.e., $g$ has at least two
squares, then we have $\rho(X) = 1$.
\item
If $t = 1$ holds, i.e., $g$ has squares,
then we have $\rho(X) \le 2$.
\item
If $t = 0$ holds, i.e., $g$ has no squares,
then we have $\rho(X) \le 3$.
\item
If $\rho(X) = 3$ holds, then we have $t=0$ and 
$X$ is $\QQ$-factorial.
\end{enumerate}
\end{proposition}

\begin{proof}
We have $X = X(q,t,m,u)$ with $m=0$ and,
according to Proposition~\ref{prop:ample}, 
we may assume that $u$ is the anticanonical 
class.  
Propositions~\ref{prop:antican}
tells us that in~$K_\QQ$, we have
$$ 
u 
\ = \ 
\frac{q+t-2}{2} \deg(g).
$$
A face $\gamma_0 \preceq \gamma$ is $X$-relevant 
if and only if it satisfies the conditions 
of Remark~\ref{rem:Xfaces} and one has 
$u \in \relint(Q(\gamma_0))$. 
For Assertions~(i), (ii) and~(iii), we consider 
the following $X$-relevant faces of $\gamma$: 
$$ 
\text{(i) }
\gamma' = \gamma_{q+1,q+2},
\qquad 
\text{(ii) }
\gamma' = \gamma_{1,2,q+1},
\qquad
\text{(iii) }
\gamma' = \gamma_{1,2,3,4}.
$$
We have $\dim(Q(\gamma')) \le 1,2,3$ according to 
the cases, because homogeneity of $g$ yields 
the following linear relations in the 
respective images~$Q(\lin_\QQ(\gamma_0))$:
$$ 
2w_{q+1} \ = \  2w_{q+2},
\qquad
w_1+w_2 \ = \ w_{q+1},
\qquad
w_1+we_2  \ = \ we_3+w_4.  
$$
The first three assertions thus follow from
the description of the Picard group provided by 
Proposition~\ref{prop:picard}: in each of the three 
cases, we have
$$ 
\Pic(X) 
\ = \ 
\bigcap_{\gamma_0 \in \cov(X)} Q(\lin_\ZZ(\gamma_0))
\ \subseteq \ 
Q(\gamma').
$$

In order to prove~(iv), assume $\rho(X) = 3$.
Assertions~(i) and~(ii) yield $t=0$. 
To obtain $\QQ$-factoriality of $X$
we have to show that $K_\QQ = \Cl(X)_\QQ$ 
is of dimension three.
The assumption $\rho(X) = 3$ together with 
Proposition~\ref{prop:picard} yields 
$\dim(Q(\gamma_0)) \ge 3$ for all 
$X$-relevant faces $\gamma_0 \preceq \gamma$.
Consider the faces 
$$ 
\gamma(i,j) 
\ \sei \ 
\gamma_{i,i+1,j,j+1}
\ \preceq \ 
\gamma,
$$
where $i,j$ are odd with $1 \le i < j \le q-1$.
These are all $X$-relevant and $Q(\gamma(i,j))$ 
is of dimension three.
Using again $\rho(X) = 3$ and  
Proposition~\ref{prop:picard},
we conclude that the $Q(\gamma(i,j))$ generate 
all the same 3-dimensional vector subspace 
$V \subseteq K_\QQ$. 
Thus $\dim(K_\QQ) = 3$ follows from
$$ 
K_\QQ
\ = \ 
Q(\QQ^{n+m})
\ = \ 
Q(\lin_\QQ(\gamma(1,3)) + \ldots + \lin_\QQ(\gamma(q-3,q)))
\ = \ 
V.
$$
\end{proof}

\goodbreak

\begin{corollary}
Let $X$ be a Fano full intrinsic quadric.
Then $\rho(X) \le 3$ holds and if we have 
$\rho(X) = 3$, then $X$ is $\QQ$-factorial.
\end{corollary}

We start our preparations of the  
proof of Theorem~\ref{thm:fullFanos}.
When performing a renumeration of variables,
we always keep $g = g_{q,t}$ a standard $K$-homogeneous 
quadratic polynomial, that means that our
renumberings respect monomials and take place only 
inside the $q$-, $t$- and $m$-blocks. 
Moreover, when visualizing the situation, 
we draw (parts of) the intersection of 
$\Eff(X) = Q(\gamma)$ with an affine hyperplane 
passing orthogonally through an inner vector
of $\Eff(X)$ and we will indicate the 
ray through, for instance, $w_i$ by a dot 
with label $w_i$.

\begin{lemma}
\label{lem:123renumb135}
Let $X = X(q,t,m,u)$ be a $\QQ$-factorial 
standard intrinsic quadric of Picard number 
three with $q \ge 4$.
If there is an $\ell$ with $5 \le \ell \le q$ 
or $n+1 \le \ell \le n+m$ such that
$\gamma_{1,2,\ell} \preceq \gamma$ is $Z$-relevant,
then there are $1 \le i \le 2$ and $3 \le j \le 4$ 
such that $\gamma_{i,j,\ell} \preceq \gamma$ 
is $X$-relevant.
\end{lemma}

\begin{proof}
Let $l_{12}$ be a linear form on $K_\QQ = \Cl_\QQ(X)$
with $l_{12}(w_1) = l_{12}(w_2) = 0$ 
and $l_{12}(w_\ell) \ge 0$.
Then $l_{12} (\deg(g)) =0 $ holds
and we may assume $l_{12}(w_3) \le 0$.
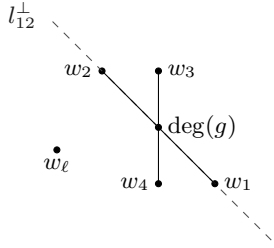
\begin{center}
\begin{tikzpicture}[scale=1.5]

   
    \coordinate (w1) at (1,0);
    \node[right] at (w1) {\small{$w_1$}};
    \node[draw,circle,inner sep=.8pt,fill=black] at (w1) {}; 
    
    \coordinate (w3) at (0.5,1);
    \node[right] at (w3) {\small{$w_3$}};
    \node[draw,circle,inner sep=.8pt,fill=black] at (w3) {}; 
    
    \coordinate (dg) at (0.5,0.5);
    \node[right] at (dg) {\small{$\deg(g)$}};
    \node[draw,circle,inner sep=.8pt,fill=black] at (dg) {};

    \coordinate (w2) at (0,1);
    \node[left] at (w2) {\small{$w_2$}};
    \node[draw,circle,inner sep=.8pt,fill=black] at (w2) {};     
 
    \coordinate (w4) at (0.5,0);
    \node[left] at (w4) {\small{$w_4$}};
    \node[draw,circle,inner sep=.8pt,fill=black] at (w4) {}; 
       
    \coordinate (w6) at (-.4,.3);
    \node[below] at (w6) {\small{$w_\ell$}};
    \node[draw,circle,inner sep=.8pt,fill=black] at (w6) {};     
    
    
    \draw[black!70!, dashed] (1.5,-.5) -- (-.45,1.45);
    \fill[black!90!]  (-.5,1.5) node [left]  {\small{$l_{12}^\perp$}};
   
    \draw[thin] (w2) -- (w1);
     \draw[thin] (w4) -- (w3);
             
\end{tikzpicture}  
\end{center}
Moreover, $Q(\gamma_{1,2,\ell})$
is contained in 
$Q(\gamma_{1,3,\ell}) \cup Q(\gamma_{2,3,\ell})$.
Thus, $u$ lies in the relative interior 
of a face $\tau$ of one of the latter 
two cones.
We have  $\tau = Q(\gamma_0)$ 
with a face~$\gamma_0$ 
of $\gamma_{1,3,\ell}$ or 
$\gamma_{2,3,\ell}$.
Since $\gamma_0$ is an 
$\overline{X}$-face,
Proposition~\ref{prop:Qfact}
yields that $Q(\gamma_0)$ must 
be of dimension three.
Thus,
$\gamma_0$ equals $\gamma_{1,3,\ell}$ or 
$\gamma_{2,3,\ell}$.
\end{proof}

\begin{lemma}
\label{lem:12renumb135}
Let $X = X(q,t,m,u)$ be a $\QQ$-factorial
standard intrinsic quadric of Picard
number three with $q \ge 6$.
If $\gamma_{1,2} \preceq \gamma$ is $Z$-relevant,
then $\gamma_{i,3,j}$ is $X$-relevant 
for some $1 \le i \le 2$ and $5 \le j \le 6$.
\end{lemma}

\begin{proof}
Let $l_{12}$ be a linear form on $K_\QQ$ 
with $l_{12}(w_1) = l_{12}(w_2) = 0$
and $l_{12}(w_3) \le 0$.
We may assume $l_{12}(w_5) \ge 0$
and then have
$Q(\gamma_{1,2}) \subseteq 
Q(\gamma_{1,3,5}) \cup Q(\gamma_{2,3,5})$.
As in the proof before, we 
conclude that $u$ lies in the 
relative interior of 
$Q(\gamma_{1,3,5})$ or that 
of $Q(\gamma_{2,3,5})$.
\end{proof}

\begin{lemma}
\label{lem:135full}
Let $X = X(q,t,m,u)$ be a $\QQ$-factorial
full standard intrinsic quadric of Picard 
number three.
Then~we have $q \ge 6$ and,
after suitable renumbering of variables, 
$\gamma_{1,3,5}$ is $X$-relevant. 
\end{lemma}

\begin{proof}
As the moving cone of~$X$ is of
dimension three, we must have
$q \ge 6$, use Proposition~\ref{prop:ample}.
The effective cone of~$X$ is generated 
by $w_1, \ldots, w_q$.
Thus, Carath\'{e}odory's theorem 
yields a $Z$-relevant face~$\tau \preceq \gamma$
generated by at most three of $e_1, \ldots, e_q$.
Suitably renumbering the variables, we achieve
$\tau\preceq\gamma_{1,3,5}$
or $\tau\preceq\gamma_{1,2,3}$.
Since all rays of $\tau$ are $X$-relevant, 
Proposition~\ref{prop:Qfact} shows that
$\tau$ is at least of dimension two.
If $\dim(\tau)=2$ holds, then 
Proposition~\ref{prop:Qfact} 
yields that $\tau$ is not an $\overline{X}$-face
which means~$\tau=\gamma_{1,2}$.
In this case, Lemma~\ref{lem:12renumb135} 
gives the assertion.
If $\tau$ is three-dimensional, then 
Lemma~\ref{lem:123renumb135}
completes the proof.
\end{proof}

\goodbreak

\begin{lemma}
\label{lem:nosmootharea2} 
Let $X = X(q,t,m,u)$ be a standard 
intrinsic quadric of Picard number three.
If there are pairwise different odd integers
$1\le a,b,c \le q-1$ such that $\tau_0\sei \gamma_{a,b,c}$
and $\tau_1\sei\gamma_{a,b,c+1}$ 
are $X$-relevant, 
then $X$ is not locally factorial.
\end{lemma}

\begin{proof}
Assume that $X$ is locally factorial.
Applying Proposition~\ref{prop:smooth}~(iii) 
to~$\tau_0$ gives $\Cl(X) \cong \ZZ^3$.
Using $K$-homogeneity of $g$ and 
suitable coordinates on $K = \Cl(X)$,
we achieve
$$
[w_a,w_{a+1},w_b,w_{b+1},w_c,w_{c+1}]
\ = \
\left[
\begin{array}{cc|cc|cc}
1 & d_1-1 & 0 & d_1   & 0 & d_1 \\
0 & d_2   & 1 & d_2-1 & 0 & d_2 \\
0 & d_3   & 0 & d_3   & 1 & d_3 -1
\end{array}
\right],
$$
where $(d_1,d_2,d_3) = \deg(g)$.
Applying Proposition~\ref{prop:smooth}~(iii) 
to~$\tau_1$ gives $d_3\in\{0,2\}$.
Let $\tau_2\sei\gamma_{a,a+1,b,b+1}$.
Since $Q(\tau_0)^\circ\cap Q(\tau_1)^\circ$ 
is three-dimensional and contained in $Q(\tau_2)$,
we conclude that the $\overline{X}$-face $\tau_2$
is $X$-relevant. 
Proposition~\ref{prop:smooth} applied to $\tau_2$
yields $d_3 \in \{-1,1\}$; a contradiction.
\end{proof}

\begin{lemma}
\label{lem:deg135}
Let $X = X(q,t,m,u)$ be a 
locally $\QQ$-factorial
standard intrinsic quadric 
of Picard number three with $q \ge 6$.
Then, after suitably renumbering the 
variables,
$\gamma_{1,3,5}$ is $X$-relevant and
$\deg(g)\in Q(\gamma_{1,3,5})$
holds.
\end{lemma}

\begin{proof}
According to Lemma~\ref{lem:135full},
we may assume that $\gamma_{1,3,5}$ 
is $X$-relevant.
If $\deg(g)$ is contained in $Q(\gamma_{1,3,5})$, 
then we are done.
Otherwise, suitably renumbering the variables
once more, we arrive in one of the following 
situations:
\medskip
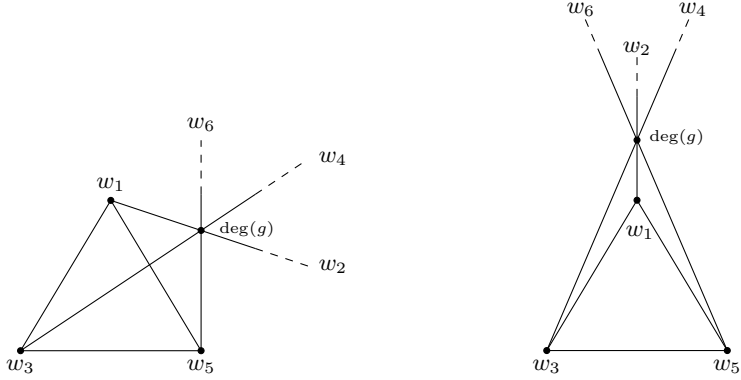
\begin{center}

\begin{tikzpicture}[scale=0.8]

    \coordinate (w1) at (-1.5, .5);
    \node[above] at (w1) {\small{$w_1$}};
    \node[draw,circle,inner sep=.8pt,fill=black] at (w1) {}; 
        
    \coordinate (w3) at (-3, -2);
    \node[below] at (w3) {\small{$w_3$}};
    \node[draw,circle,inner sep=.8pt,fill=black] at (w3) {}; 
    
    \coordinate (w5) at (0, -2);
    \node[below] at (w5) {\small{$w_5$}};
    \node[draw,circle,inner sep=.8pt,fill=black] at (w5) {}; 
    
    \coordinate (w6) at (0, 1);
    \node[above] at ($1.5*(w6)$) {\small{$w_6$}};
    
    \coordinate (w2) at (1.5, -.5);
    \node[right] at ($1.2*(w2)$) {\small{$w_2$}};

    \coordinate (w4) at (1.5, 1);
    \node[right] at ($1.2*(w4)$) {\small{$w_4$}};
    
    \coordinate (degg) at (0,0);
    \node[right] at ($(degg)+(0.15,0.02)$) {\tiny{$\deg(g)$}};
    \node[draw,circle,inner sep=.8pt,fill=black] at (degg) {};     

    \draw[thin]  (w1) -- (w3);
    \draw[thin]  (w3) -- (w5);
    \draw[thin]  (w1) -- (w5);
    
    \draw [thin] (w1) --  ($.6*(w2)$);
    \draw [thin, dashed] ($.6*(w2)$) -- ($1.2*(w2)$);
 
    \draw [thin] (w3) --  ($.6*(w4)$);
    \draw [thin, dashed] ($.6*(w4)$) -- ($1.15*(w4)$);
    
    \draw [thin] (w5) --  ($.6*(w6)$);
    \draw [thin, dashed] ($.6*(w6)$) -- ($1.5*(w6)$);
          
               
  \end{tikzpicture}   
\qquad \qquad \qquad 
\begin{tikzpicture}[scale=0.8]

    \coordinate (w1) at (0, -1);
    \node[below] at ($(w1)+(.05,-.27)$) {\small{$w_1$}};
    \node[draw,circle,inner sep=.8pt,fill=black] at (w1) {}; 
        
    \coordinate (w3) at (-1.5, -3.5);
    \node[below] at (w3) {\small{$w_3$}};
    \node[draw,circle,inner sep=.8pt,fill=black] at (w3) {}; 
    
    \coordinate (w5) at (1.5, -3.5);
    \node[below] at (w5) {\small{$w_5$}};
    \node[draw,circle,inner sep=.8pt,fill=black] at (w5) {}; 
    
    \coordinate (w6) at (-1.5, 3.5);
    \node[] at ($.62*(w6)$) {\small{$w_6$}};
    
    \coordinate (w2) at (0, 1.2);
    \node[above] at ($1.05*(w2)$) {\small{$w_2$}};
    
    \coordinate (w4) at (1.5, 3.5);
    \node[] at ($.62*(w4)$) {\small{$w_4$}};
    
    \coordinate (degg) at (0,0);
    \node[right] at ($(degg)+(0.05,0.05)$) {\tiny{$\deg(g)$}};
    \node[draw,circle,inner sep=.8pt,fill=black] at (degg) {};

    \draw[thin]  (w1) -- (w3);
    \draw[thin]  (w3) -- (w5);
    \draw[thin]  (w1) -- (w5);
    
    \draw [thin] (w1) --  ($.6*(w2)$);
    \draw [thin, dashed] ($.6*(w2)$) -- ($1.3*(w2)$);
 
    \draw [thin] (w3) --  ($.4*(w4)$);
    \draw [thin, dashed] ($.4*(w4)$) -- ($.6*(w4)$);
    
    \draw [thin] (w5) --  ($.4*(w6)$);
    \draw [thin, dashed] ($.4*(w6)$) -- ($.6*(w6)$);
          
               
  \end{tikzpicture}    

\end{center}
\medskip

\noindent
In the r.h.s.~setting, exchanging~$T_1$ and~$T_2$ 
yields the assertion.
So, consider the l.h.s.~setting.
Applying Lemma~\ref{lem:nosmootharea2} to 
the $X$-relevant face~$\gamma_{1,3,5}$ 
yields that neither~$\gamma_{1,3,6}$
nor $\gamma_{2,3,5}$ is $X$-relevant. 
Note that we have
$$
u
\ \in \
Q(\gamma_{1,3,5})^\circ
\ \subseteq \
Q(\gamma_{1,3,6}) \ \cup\
Q(\gamma_{2,3,6})\ \cup\
Q(\gamma_{2,3,5})
$$
and that all faces of $\gamma_{1,3,6}, \gamma_{2,3,6}$ 
and~$\gamma_{2,3,5}$
are $\overline{X}$-faces.
Proposition~\ref{prop:Qfact} shows 
that~$\gamma_{2,3,6}$ is $X$-relevant.
After exchanging $T_1$ and $T_2$ as well as $T_5$ and $T_6$,
the new $\gamma_{1,3,5}$ is $X$-relevant
and we have $\deg(g) \in Q(\gamma_{1,3,5})$.
\end{proof}

\begin{lemma}
\label{lem:nosmootharea1}
Let $X = X(q,t,m,u)$ be a standard 
intrinsic quadric of Picard number three.
If there are pairwise different odd integers
$1\le a,b,c\le q-1$ such that $\tau_0\sei \gamma_{a,b,c}$
and $\tau_1\sei\gamma_{a+1,b+1,c+1}$ are $X$-relevant, 
then $X$ is not locally factorial.
\end{lemma}

\begin{proof}
Assume that $X$ is locally factorial.
Applying Proposition~\ref{prop:smooth}~(iii) 
to~$\tau_0$ gives $\Cl(X) \cong \ZZ^3$.
Using $K$-homogeneity of $g$ and 
suitable coordinates on $K = \Cl(X)$,
we achieve
$$
[w_a,w_{a+1},w_b,w_{b+1},w_c,w_{c+1}]
\ = \
\left[
\begin{array}{cc|cc|cc}
1 & d_1-1 & 0 & d_1   & 0 & d_1 \\
0 & d_2   & 1 & d_2-1 & 0 & d_2 \\
0 & d_3   & 0 & d_3   & 1 & d_3-1
\end{array}
\right],
$$
where $(d_1,d_2,d_3) = \deg(g)$. 
Consider 
$\tau_{i,j} \sei \cone(e_i,e_{i+1},e_j,e_{j+1})$,
where $i,j \in\{a,b,c\}$ with 
$i \neq j$.
For all three possibilities, we have  
$$
Q(\tau_0)^\circ\cap Q(\tau_1)^\circ
\ \subseteq \
Q(\tau_{i,j})^\circ.
$$
Thus, all the $\tau_{i,j}$ are $X$-relevant.
Proposition~\ref{prop:smooth}~(iii) says 
that $w_i,w_{i+1},w_j,w_{j+1}$ generate $K$ 
in all cases, which implies 
$d_1,d_2,d_3 \in \{-1,1\}$. 
Consequently,
$$ 
\det(w_{a+1},w_{b+1},w_{c+1})
\ = \ 
d_1 + d_2 + d_3 -1 
\ \in \
\{0,2\}.
$$
But Proposition~\ref{prop:smooth}~(iii)
applied to  $\tau_1 = \gamma_{a+1,b+1,c+1}$
shows that this determinant should equal 
$\pm 1$; a contradiction.
\end{proof}

\begin{lemma}
\label{lem:firstsixweights}
Let $X = X(q,t,m,u)$ be a locally factorial
full standard intrinsic quadric of Picard 
number three.
Then $K = \Cl(X) \cong \ZZ^3$ and $q \ge 6$ hold.
Moreover, by a suitable renumeration of variables
we achieve
$$
[w_1,\ldots,w_6]
\ = \
\left[
\begin{array}{cc|cc|cc}
1 & d_1-1 & 0 & d_1 & 0 & d_1 \\
0 &     1 & 1 &   0 & 0 &   1 \\
0 &     1 & 0 &   1 & 1 &   0 \\
\end{array}
\right],
$$
where $d_1 \in\ZZ_{\ge 0}$,
the faces
$\gamma_{1,3,5}$, $\gamma_{1,4,6}$,
$\gamma_{1,2,3,4}$, $\gamma_{1,2,5,6}$
are all $X$-relevant
and, moreover, 
$u \in \cone(w_1,w_3,d) \cap Q(\gamma_{1,4,6})^\circ$
holds,
where $d=(d_1,1,1) = \deg(g)$.
In particular, we have $t=0$ and 
$n=q$ is even.
\end{lemma}

\begin{proof}
Lemmas~\ref{lem:135full} and~\ref{lem:deg135}
show that $q \ge 6$ holds and that after suitably 
renumbering the variables,
$\gamma_{1,3,5}$ is an $X$-relevant face 
with~$d = \deg(g) \in Q(\gamma_{1,3,5})$.
By Proposition~\ref{prop:smooth}~(iii),
the cone $Q(\gamma_{1,3,5})$ is of dimension three
and $w_1,w_3,w_5$ freely generate $K = \ZZ^3$.

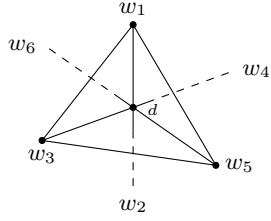
\begin{center}
\begin{tikzpicture}[scale=1.1]

    \coordinate (w1) at (0, 1);
    \node[above] at (w1) {\small{$w_1$}};
    \node[draw,circle,inner sep=.8pt,fill=black] at (w1) {}; 
        
    \coordinate (w3) at (-1.1, -.4);
    \node[below] at (w3) {\small{$w_3$}};
    \node[draw,circle,inner sep=.8pt,fill=black] at (w3) {}; 
    
    \coordinate (w5) at (1, -.7);
    \node[below, right] at (w5) {\small{$w_5$}};
    \node[draw,circle,inner sep=.8pt,fill=black] at (w5) {}; 
    
    \coordinate (w6) at (-1, .7);
    \node[left] at ($1.05*(w6)$) {\small{$w_6$}};
    
    \coordinate (w2) at (0, -.7);
    \node[below] at ($(w2)+(0,-.3)$) {\small{$w_2$}};

    \coordinate (w4) at (1.1, .4);
    \node[right] at ($1.1*(w4)$) {\small{$w_4$}};
    
    \coordinate (degg) at (0,0);
    \node[right] at ($(degg)+(0.06,-.025)$) {\tiny{$d$}};
    \node[draw,circle,inner sep=.8pt,fill=black] at (degg) {};     

    \draw[thin]  (w1) -- (w3);
    \draw[thin]  (w3) -- (w5);
    \draw[thin]  (w1) -- (w5);
    
    \draw [thin] (w1) --  ($.4*(w2)$);
    \draw [thin, dashed] ($.4*(w2)$) -- ($1.4*(w2)$);
 
    \draw [thin] (w3) --  ($.15*(w4)$);
    \draw [thin, dashed] ($.15*(w4)$) -- ($1.1*(w4)$);
    
    \draw [thin] (w5) --  ($.15*(w6)$);
    \draw [thin, dashed] ($.15*(w6)$) -- ($1.05*(w6)$);
            
  \end{tikzpicture}   
\end{center}

\noindent
Note that $d$ might as well lie on the boundary 
of $Q(\gamma_{1,3,5})$.
However, $u$ lies in the relative interior of 
$Q(\gamma_{1,3,5})$ and, suitably renumbering
the variables $T_1, \ldots, T_6$, we achieve 
that $\tau \sei \cone(w_1,w_3,d)$ satisfies
$$
\dim(\tau) \ = \ 3,
\qquad \qquad
u 
\ \in \ 
\tau \setminus Q(\gamma_{1,3}).
$$
Note that we have $w_2 \not\in \cone(w_1,d)$ and 
$w_4 \not\in \cone(w_3,d)$, because 
otherwise $w_1$ or~$w_3$ would lie on $\cone(d)$, 
contradicting $\dim(\tau) = 3$.
We conclude $u \in Q(\gamma_{1,2,3,4})^\circ$ 
and thus~$\gamma_{1,2,3,4}$ is $X$-relevant.
Observe 
$$
\tau
\ \subseteq \
Q(\gamma_{1,5,6}) \cup Q(\gamma_{3,5,6}) \cup Q(\gamma_{1,3,6}).
$$
As $X$ is locally factorial and $\gamma_{1,3,5}$ is $X$-relevant,
Lemma~\ref{lem:nosmootharea2} shows $u \not\in Q(\gamma_{1,3,6})^\circ$.
Thus, $u$ lies in one of the other two r.h.s.~cones.
Suitably renumbering $T_1, \ldots, T_4$, we 
achieve $u \in Q(\gamma_{1,5,6})$.
Then also $u \in Q(\gamma_{1,2,5,6})$ holds.
Now, 
$$
\gamma_{1,2,5,6}, 
\quad
\gamma_{1,5}, 
\gamma_{1,6}, 
\gamma_{2,5}, 
\gamma_{2,6},
\quad 
\gamma_{1}, 
\gamma_{2}, 
\gamma_{5}, 
\gamma_{6}
$$ 
are $\overline{X}$-faces. 
Thus, Proposition~\ref{prop:Qfact} yields
that $u$ does not lie in any of the corresponding
$Q(\gamma_i)$ and $Q(\gamma_{i,j})$. 
Consequently, $Q(\gamma_{1,2,5,6})$
is three-dimensional an contains $u$ in its 
relative interior.
That means that $\gamma_{1,2,5,6}$ is $X$-relevant.
Observe
$$
u 
\ \in \
Q(\gamma_{1,5,6}) \cap \tau
\ \subseteq \
Q(\gamma_{1,4,6}) \cup Q(\gamma_{2,4,6}).
$$
Applying Lemma~\ref{lem:nosmootharea1} 
to~$\gamma_{1,3,5}$,
we see that~$\gamma_{2,4,6}$ is not $X$-relevant. 
Moreover, all faces of the two cones $\gamma_{1,4,6}$
and $\gamma_{2,4,6}$ are $\overline{X}$-relevant.
We conclude $u \in Q(\gamma_{1,4,6})^\circ$ 
and thus $\gamma_{1,4,6}$ is $X$-relevant.
A suitable choice of coordinates on $K = \ZZ^3$ yields
$$
[w_1,\ldots,w_{6}]
\ = \
\left[
\begin{array}{cc|cc|cc}
1 & d_1-1 & 0 & d_1   & 0 & d_1 \\
0 & d_2   & 1 & d_2-1 & 0 & d_2 \\
0 & d_3   & 0 & d_3   & 1 & d_3-1
\end{array}
\right].
$$
Because of $d \in Q(\gamma_{1,3,5}) = \QQ^3_{\ge0}$, 
we have $d_1,d_2, d_3\ge 0$.
Proposition~\ref{prop:smooth}~(iii) together with the
$X$-relevant faces $\gamma_{1,2,3,4}$ 
and $\gamma_{1,2,5,6}$ 
show that $d_2=d_3=1$ holds. 
\end{proof}

\begin{proof}[Proof of Theorem~\ref{thm:fullFanos}]
Proposition~\ref{prop:picardnumber1}
and Theorem~\ref{thm:smoothrhoX2} 
settle the case of Picard number at most 
two.
Proposition~\ref{prop:fullFano}
settles the case of Picard number at least four.
The remaining task is to
consider smooth full intrinsic 
quadrics $X$ of Picard number three.
By Proposition~\ref{prop:iq2siq}
we may assume that $X = X(q,t,m,u)$ is a standard
intrinsic quadric.
Moreover, Lemma~\ref{lem:firstsixweights} says 
$\Cl(X) = \ZZ^3$ and that by a suitable choice 
of coordinates, we have 
$$
[w_1,\ldots,w_6]
\ = \
\left[
\begin{array}{cc|cc|cc}
1 & d_1-1 & 0 & d_1 & 0 & d_1 \\
0 &     1 & 1 &   0 & 0 &   1 \\
0 &     1 & 0 &   1 & 1 &   0 \\
\end{array}
\right],
$$
where $d_1 \in \ZZ_{\ge 0}$
and the ample cone of~$X$ is 
contained in $Q(\gamma_{146})^\circ$.
Proposition~\ref{prop:antican}
tells us that~$-\mathcal{K}_X$ is a multiple 
of~$\deg(g) = (d_1,1,1)$.
But $(d_1,1,1)$ can't be represented as a 
strict positive combination over
$w_1$, $w_4$ and $w_6$.
Thus, $\deg(g)$ is not contained in 
$Q(\gamma_{146})^\circ$.
Consequently, $-\mathcal{K}_X$ is not 
ample and hence $X$ is not Fano.
\end{proof}

\section{Proof of Theorem~\ref{thm:fullpic3}}

A detailed analysis of the combinatorics of 
the $X$-relevant faces together with the 
resulting conditions on determinants
provided by Proposition~\ref{prop:smooth}
leads to the normal form asserted in the Theorem.
This is the first part of the proof.
The second one establishes the geometric 
supplements.
At the end of the section, we prove 
Corollary~\ref{cor:fujitapic3full}.

\begin{proof}[Proof of Theorem~\ref{thm:fullpic3}, Part~I]
According to Proposition~\ref{prop:iq2siq}
we may assume that $X = X(q,t,m,u)$ 
is a standard intrinsic quadric.
Lemma~\ref{lem:firstsixweights} 
tells us $n = q \ge 6$ and $K = \ZZ^3$.
Moreover, choosing suitable coordinates 
on $K$, we achieve that
$$
[w_1,\ldots,w_6]
\ = \
\left[
\begin{array}{cc|cc|cc}
1 & d_1-1 & 0 & d_1 & 0 & d_1 \\
0 &     1 & 1 &   0 & 0 &   1 \\
0 &     1 & 0 &   1 & 1 &   0 \\
\end{array}
\right]
$$
holds with $d_1 \in \ZZ_{\ge 0}$,
the faces $\gamma_{1,3,5},\gamma_{1,4,6},
\gamma_{1,2,3,4},\gamma_{1,2,5,6}$
of $\gamma$ are all $X$-relevant
and the ample class~$u$ of $X$ satisfies
$$
u 
\ \in \ 
\cone(w_1,w_3,d) \cap Q(\gamma_{1,4,6})^\circ,
$$
where~$d=(d_1,1,1)$ denotes the degree of 
the relation~$g$. 
Depending on~$d_1$, the situation looks 
as follows:
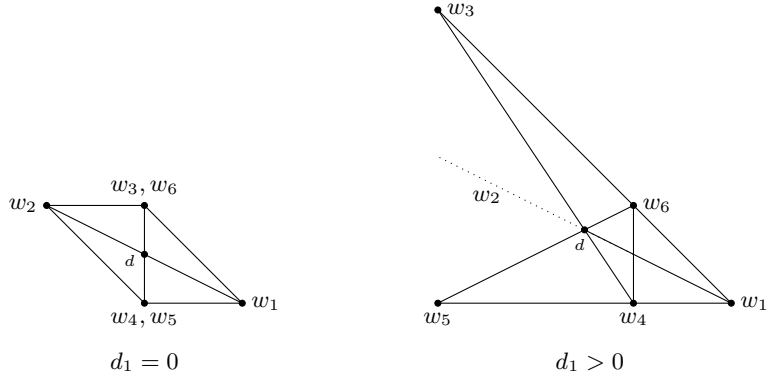
\begin{center}
\begin{tikzpicture}[scale=1.3]

    \coordinate (w1) at (1,0);
    \node[right] at (w1) {\small{$w_1$}};
    \node[draw,circle,inner sep=.8pt,fill=black] at (w1) {}; 
    
    \coordinate (w6) at (0,1);
    \node[above] at (w6) {\small{$w_3, w_6$}};
    \node[draw,circle,inner sep=.8pt,fill=black] at (w6) {}; 
    
    \coordinate (w2) at (-1,1);
    \node[left] at (w2) {\small{$w_2$}};
    \node[draw,circle,inner sep=.8pt,fill=black] at (w2) {}; 
    
    \coordinate (w4) at (0, 0);
    \node[below] at ($(w4)$) {\small{$w_{4},w_5$}};
    \node[draw,circle,inner sep=.8pt,fill=black] at (w4) {}; 
    
    
    \coordinate (le) at (0,-.6);
    \node[] at (le) {\small{$d_1=0$}};
    
    \coordinate (degg) at (0,.5);
    \node[left] at ($(degg)+(-.005,-0.06)$) {\tiny{$d$}};
    \node[draw,circle,inner sep=.8pt,fill=black] at (degg) {};  
  
    \draw[thin]  (w1) -- (w6);
    \draw[thin]  (w1) -- (w2);
    \draw[thin]  (w1) -- (w4);
    \draw[thin]  (w4) -- (w6);
    \draw[thin]  (w2) -- (w4);
    \draw[thin]  (w2) -- (w6);


    \coordinate (w1) at (6,0);
    \node[right] at (w1) {\small{$w_1$}};
    \node[draw,circle,inner sep=.8pt,fill=black] at (w1) {}; 
        
    \coordinate (w3) at (3,3);
    \node[above, right] at (w3) {\small{$w_3$}};
    \node[draw,circle,inner sep=.8pt,fill=black] at (w3) {}; 
    
    \coordinate (w5) at (3,0);
    \node[below] at (w5) {\small{$w_5$}};
    \node[draw,circle,inner sep=.8pt,fill=black] at (w5) {}; 
    
    \coordinate (w6) at (5,1);
    \node[above,right] at (w6) {\small{$w_6$}};
    \node[draw,circle,inner sep=.8pt,fill=black] at (w6) {}; 
    
    \coordinate (w2) at (3.5,1.25);
    \node[left, below] at (w2) {\small{$w_2$}};
    
    \coordinate (w4) at (5,0);
    \node[below] at ($(w4)$) {\small{$w_4$}};
    \node[draw,circle,inner sep=.8pt,fill=black] at (w4) {}; 
    
    \coordinate (degg) at (4.5,.75);
    \node[below] at ($(degg)+(-.05,0)$) {\tiny{$d$}};
    \node[draw,circle,inner sep=.8pt,fill=black] at (degg) {};  
    
  
    \coordinate (le) at (5,-.6);
    \node[left] at (le) {\small{$d_1> 0$}};
    
    \draw[thin]  (w1) -- (w3);
    \draw[thin]  (w1) -- (degg);
    \draw[thin]  (w1) -- (w5);
    \draw[thin]  (w4) -- (w6);
    \draw[thin, dotted]  (5,.5) -- (3,1.5);
    
    \draw [thin] (w3) -- (w4);
    
    \draw [thin] (w5) -- (w6);

\end{tikzpicture}   
\end{center}

\noindent
We claim that $n \ge 8$ holds.
Otherwise, $n=6$ and
according to Proposition~\ref{prop:ample},
we have 
$$ 
\Mov(X) 
\ \subseteq \
\cone(w_1,w_3,w_5) \cap \cone(w_2,w_4,w_6).
$$
By Lemma~\ref{lem:nosmootharea1},
this contradicts smoothness of~$X$.
Thus, we obtain $n\ge 8$, which implies 
in particular~$\dim(X)\ge 4$.

We specify the possible positions of the weights~$w_\ell$,
where $\ell = 7, \ldots, n$.
For $i = 1, \ldots, 6$ choose linear forms
$l_{i}$ on $K_\QQ$ such that
$$
l_{i}(w_i) = l_{i}(u) = 0, \quad i=1, \ldots, 6, 
$$
$$
l_{i}(w_1) > 0, \ i=3, \ldots, 6, 
\quad 
l_{2}(w_4) > 0, 
\quad
l_{1}(w_3) > 0.
$$
Each of the linear forms $l_1,\ldots, l_6$
defines a negative half space and a positive 
half space:
$$
H_i^- 
\ \sei \
\{x \in K_\QQ; \; l_i(x) \le 0\},
\qquad
H_i^+ 
\ \sei \
\{x \in K_\QQ; \; l_i(x) \ge 0\}.
$$
Note that $\gamma_{i,\ell}$ is an $\overline{X}$-face 
for all $i=1,\ldots,6$ and $\ell = 7, \ldots, n$. 
Thus, Proposition~\ref{prop:Qfact} yields
that $u$ can't lie in $Q(\gamma_{i,\ell})$.
In other words, for all $i=1,\ldots,6$ and $\ell = 7, \ldots, n$,
we have
$$ 
w_\ell 
\ \not\in \
H_i 
\ \sei \ 
\QQ u - \QQ_{\ge 0} w_i
\ \subseteq \
H_i^- \cap H_i^+.
$$
The half planes $H_1, \ldots, H_6$ define 
a subdivision of $K_\QQ = \QQ^3$ into the 
following three-dimensional cones, all having 
$\QQ u$ as a common line:
$$
\begin{array}{ccccc}
M_a \sei H_1 + H_4, 
& &
M_b \sei H_4 + H_5,
& &
M_c \sei H_5 + H_2,
\\[1ex]
M_d \sei H_2 + H_3, 
& &
M_e \sei H_3 + H_6, 
& &
M_f \sei H_6 + H_1.
\end{array}
$$
As observed before, the degrees $w_\ell$, 
where $\ell = 7, \ldots, n$, are distributed
over the relative interiors 
$M^\circ_a, \ldots, M^\circ_f$.
According to the cases $d_1=0$ and $d_1 > 0$,
the situation looks as follows.


\medskip
\noindent
$(d_1=0)$

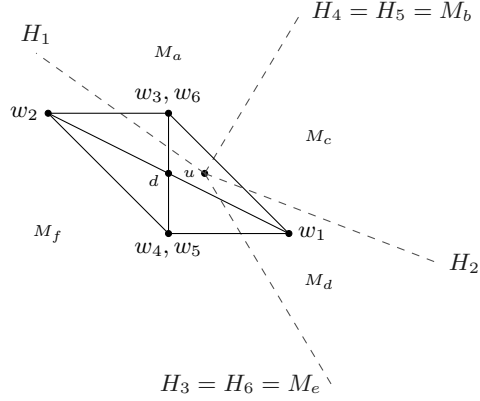
\hspace{3cm}
\begin{tikzpicture}[scale=1.6]

    \coordinate (w1) at (1,0);
    \node[right] at (w1) {\small{$w_1$}};
    \node[draw,circle,inner sep=.8pt,fill=black] at (w1) {}; 
    
    \coordinate (w6) at (0,1);
    \node[above] at (w6) {\small{$w_3, w_6$}};
    \node[draw,circle,inner sep=.8pt,fill=black] at (w6) {}; 
    
    \coordinate (w2) at (-1,1);
    \node[left] at (w2) {\small{$w_2$}};
    \node[draw,circle,inner sep=.8pt,fill=black] at (w2) {}; 
    
    \coordinate (w4) at (0, 0);
    \node[below] at ($(w4)$) {\small{$w_{4},w_5$}};
    \node[draw,circle,inner sep=.8pt,fill=black] at (w4) {}; 
    
    \coordinate (u) at (.3,.5);
    \node[left] at ($(u)$) {\tiny{$u$}};
    \node[draw,circle,inner sep=.8pt,fill=black] at (u) {}; 
    
    \coordinate (le) at (0,-.7);
    
    \coordinate (degg) at (0,.5);
    \node[left] at ($(degg)+(-.005,-0.06)$) {\tiny{$d$}};
    \node[draw,circle,inner sep=.8pt,fill=black] at (degg) {};  
  
    \draw[thin, black!70!, dashed] (u) --  ($3*(u)-2*(w1)$);
    \fill[black!90!] ($3*(u)-2*(w1)$) node [above]  {\small{$H_{1}$}}; 
    \draw[thin, black!70!, dashed] (u) --  ($2.5*(u)-1.5*(w2)$);
    \fill[black!90!] ($2.5*(u)-1.5*(w2)$) node [right]  {\small{$H_{2}$}}; 
    \draw[thin, black!70!, dashed] (u) --  ($3.7*(u)-2.7*(w4)$);
    \fill[black!90!] ($3.7*(u)-2.7*(w4)$) node [right]  {\small{$H_4 = H_5 = M_b$}}; 
    \draw[thin, black!70!, dashed] (u) --  ($4.5*(u)-3.5*(w6)$);
    \fill[black!90!] ($4.5*(u)-3.5*(w6)$) node [left]  {\small{$H_3 = H_6 = M_e$}};     
  
    \draw[thin]  (w1) -- (w6);
    \draw[thin]  (w1) -- (w2);
    \draw[thin]  (w1) -- (w4);
    \draw[thin]  (w4) -- (w6);
    \draw[thin]  (w2) -- (w4);
    \draw[thin]  (w2) -- (w6);
    
    \coordinate (a) at (0,1.5);
    \node[] at ($(a)$) {\tiny{$M_a$}};
    \coordinate (c) at (1.25,.8);
    \node[] at ($(c)$) {\tiny{$M_c$}}; 
    \coordinate (d) at (1.25,-.4);
    \node[] at ($(d)$) {\tiny{$M_d$}};
    \coordinate (f) at (-1,0);
    \node[] at ($(f)$) {\tiny{$M_f$}}; 
\end{tikzpicture}   


\noindent
$(d_1>0)$

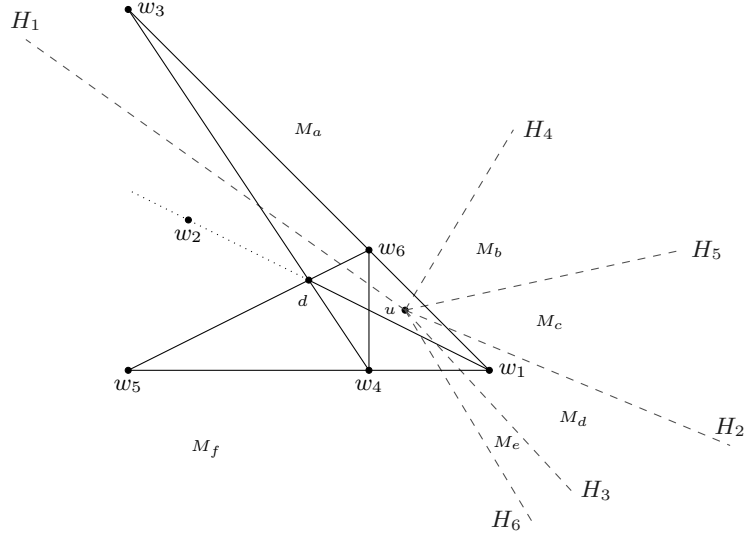
\hspace{1.75cm}
\begin{tikzpicture}[scale=1.6]


    \coordinate (w1) at (6,0);
    \node[right] at (w1) {\small{$w_1$}};
    \node[draw,circle,inner sep=.8pt,fill=black] at (w1) {}; 
     

   \coordinate (w3) at (3,3);
    \node[above, right] at (w3) {\small{$w_3$}};
    \node[draw,circle,inner sep=.8pt,fill=black] at (w3) {}; 
    
    \coordinate (w5) at (3,0);
    \node[below] at (w5) {\small{$w_5$}};
    \node[draw,circle,inner sep=.8pt,fill=black] at (w5) {}; 
    
    \coordinate (w6) at (5,1);
    \node[above,right] at (w6) {\small{$w_6$}};
    \node[draw,circle,inner sep=.8pt,fill=black] at (w6) {}; 
    
    \coordinate (w2) at (3.5,1.25);
    \node[left, below] at (w2) {\small{$w_2$}};
    \node[draw,circle,inner sep=.8pt,fill=black] at (w2) {}; 
    
    \coordinate (w4) at (5,0);
    \node[below] at ($(w4)$) {\small{$w_4$}};
    \node[draw,circle,inner sep=.8pt,fill=black] at (w4) {}; 

    \coordinate (degg) at (4.5,.75);

 
 \node[below] at ($(degg)+(-.05,-0.05)$) {\tiny{$d$}};
    \node[draw,circle,inner sep=.8pt,fill=black] at (degg) {};  
    
    \coordinate (u) at (5.3,.5);
    \node[left] at ($(u)$) {\tiny{$u$}};
    \node[draw,circle,inner sep=.8pt,fill=black] at (u) {}; 
  
    \coordinate (le) at (5,-.7);

    \draw[thin, black!70!, dashed] (u) --  ($5.5*(u)-4.5*(w1)$);
    \fill[black!90!] ($5.5*(u)-4.5*(w1)$) node [above]  {\small{$H_{1}$}}; 
    \draw[thin, black!70!, dashed] (u) --  ($2.5*(u)-1.5*(w2)$);
    \fill[black!90!] ($2.5*(u)-1.5*(w2)$) node [above]  {\small{$H_{2}$}}; 
    \draw[thin, black!70!, dashed](u) --  ($1.6*(u)-.6*(w3)$);
    \fill[black!90!] ($1.6*(u)-.6*(w3)$) node [right]  {\small{$H_{3}$}}; 
    \draw[thin, black!70!, dashed] (u) --  ($4*(u)-3*(w4)$);
    \fill[black!90!] ($4*(u)-3*(w4)$) node [right]  {\small{$H_{4}$}}; 
    \draw[thin, black!70!, dashed] (u) --  ($2*(u)-1*(w5)$);
    \fill[black!90!] ($2*(u)-1*(w5)$) node [right]  {\small{$H_{5}$}}; 
    \draw[thin, black!70!, dashed] (u) --  ($4.5*(u)-3.5*(w6)$);
    \fill[black!90!] ($4.5*(u)-3.5*(w6)$) node [left]  {\small{$H_{6}$}};     
       
    \draw[thin]  (w1) -- (w3);
    \draw[thin]  (w1) -- (degg);
    \draw[thin]  (w1) -- (w5);
    \draw[thin]  (w4) -- (w6);
    \draw [thin] (w3) -- (w4);
    \draw [thin] (w5) -- (w6);
    
    \draw[thin, dotted]  (degg) -- (3,1.5);
       
    \coordinate (a) at (4.5,2);
    \node[] at ($(a)$) {\tiny{$M_a$}};
    \coordinate (b) at (6,1);
    \node[] at ($(b)$) {\tiny{$M_b$}};  
    \coordinate (c) at (6.5,.4);
    \node[] at ($(c)$) {\tiny{$M_c$}}; 
    \coordinate (d) at (6.7,-.4);
    \node[] at ($(d)$) {\tiny{$M_d$}};
    \coordinate (e) at (6.15,-.6);
    \node[] at ($(e)$) {\tiny{$M_e$}}; 
    \coordinate (f) at (3.65,-.65);
    \node[] at ($(f)$) {\tiny{$M_f$}}; 
   
\end{tikzpicture}   

We show $w_\ell \not\in M_b$ for $\ell = 7, \ldots, n$.
Otherwise, $w_\ell = (x,y,z) \in M_b^\circ$ holds.
Then $\gamma_{1,\ell,5}$, $\gamma_{4,\ell,5}$,
$\gamma_{2,4,\ell}$ and $\gamma_{6,4,\ell}$
are $X$-relevant.
The way we list the indices $i,j,k$ 
for the $\gamma_{i,j,k}$ ensures 
that $\det(w_i,w_j,w_k)$ is positive and thus, 
by Proposition~\ref{prop:smooth}~(iii), equals one.
For $1,\ell,5$ this implies $y=1$.
Looking at $4,\ell,5$ yields $d_1=1$.
Taking $2,4,\ell$ gives $z=x$.
But this leads to $\det(w_6,w_4,w_\ell) = -1$;
a contradiction. 
  
We show $w_\ell \not\in M_e$ for $\ell = 7, \ldots, n$.
Otherwise, $w_\ell = (x,y,z) \in M_e^\circ$ holds
and thus $\gamma_{1,3,\ell}$, $\gamma_{6,3,\ell}$,
$\gamma_{6,2,\ell}$, $\gamma_{6,4,\ell}$
are $X$-relevant.
Again the indices are listed in a way that 
$\det(w_i,w_j,w_k) = 1$ holds.
For $1,3,\ell$ this means $z=1$.
Then $6,3,\ell$ brings us to $d_1 = 1$.
Taking $6,2,\ell$ yields $y=x$.
But then $\det(w_6,w_4,w_\ell) = -1$ holds; a 
contradiction.

Next observe that $\{w_\ell,w_{\ell+1}\} \not\subseteq M_f$
holds for all odd~$\ell \ge 7$,
because otherwise we had the $X$-relevant 
faces $\gamma_{1,3,\ell}$ and $\gamma_{1,3,\ell+1}$,
contradicting Lemma~\ref{lem:nosmootharea2}.
Hence, suitably renumbering the variables 
$T_7, \ldots, T_n$, we achieve 
$w_\ell \in M^\circ_a \cup M^\circ_c \cup M^\circ_d$
for all odd~$7 \le \ell < n$.
Thus, for a given odd $\ell \ge 7$, 
we have the following possibilities
for the positions of the pair $w_\ell,w_{\ell+1}$:
\begin{center}
\renewcommand{\arraystretch}{1.8} 
\begin{tabular}{c|c|c|l}
Case 
& 
Pos.~of~$w_\ell$
&
Pos.~of~$w_{\ell+1}$
&
Resulting $X$-relevant faces
\\
\hline
(1)
&
$M^\circ_a$
&
$M^\circ_d$
& 
$\gamma_{1,\ell,5}$, $\gamma_{3, 5, \ell+1}$, $\gamma_{6,4,\ell+1}$
\\
\hline
(2)
&
$M^\circ_a$
&
$M^\circ_f$
& 
$\gamma_{1,\ell,4}$, $\gamma_{1,6,\ell+1}$
\\
\hline
(3)
&
$M^\circ_c$
&
$M^\circ_f$
& 
$\gamma_{\ell,3,5}$, $\gamma_{1,3,\ell+1}$, $\gamma_{\ell,2,4}$
\\
\hline
(4)
&
$M^\circ_d$
&
$M^\circ_f$
& 
$\gamma_{\ell,3,5}$, $\gamma_{1,6,\ell+1}$, 
$\gamma_{\ell,6,4}$, $\gamma_{\ell,3,2}$,
$\gamma_{\ell+1,3,1}$
\end{tabular}
\end{center}
Here, the position of $w_{\ell+1}$ is determined 
by $w_\ell + w_{\ell +1} = \deg(g)$.
Moreover, the $\gamma_{i,j,k}$ occuring in the 
table are some but not necessarily all $X$-relevant
faces containing~$e_\ell$ or $e_{\ell+1}$ 
and the indices $i,j,k$ are listed in such a 
manner that $\det(w_i,w_j,w_k) = 1$ holds.
We now discuss each of these cases.

\medskip
\noindent
\emph{Case~(1)}. 
Write $w_\ell = (x,y,z)$.
Then $\det(w_1,w_\ell,w_5) = 1$
implies $y = 1$
and $\det(w_3,w_5,w_{\ell+1}) = 1$
gives $x = d_1-1$.
Moreover, $\det(w_6,w_4,w_{\ell+1}) = 1$
leads to $z = 1$.
Thus, we arrive at~$w_1=w_{\ell+1}$, $w_2=w_\ell$,
which contradicts $w_\ell \in M_a^\circ$.

\medskip
\noindent
\emph{Case (2)}. 
Write $w_\ell = (x,y,z)$.
Then $\det(w_1,w_\ell,w_4) = 1$
implies $y = 1$.
From $\det(w_1,w_6,w_{\ell+1})$ we 
derive $z=0$.
Thus, we obtain
$$
[w_1,w_2,w_3,w_4,w_5,w_6,w_\ell,w_{\ell+1}]
 = \
\left[
\begin{array}{cc|cc|cc|cc}
1 & d_1-1    & 0 & d_1   & 0 & d_1    & x  & d_1-x  \\
0 & 1     & 1 & 0   & 0 & 1    & 1  & 0   \\
0 & 1     & 0 & 1   & 1 & 0    & 0     & 1
\end{array}
\right].
$$
The weights are arranged as follows,
where~$w_2$ lies on the dotted line,
$w_\ell$ on the thickly dotted line
and $w_{\ell+1}$ on the dash-dotted line.
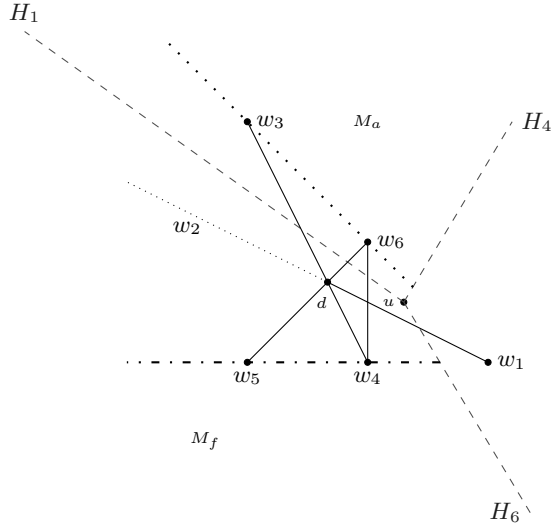
\begin{center}
\begin{tikzpicture}[scale=1.6]

    \coordinate (w1) at (6,0);
    \node[right] at (w1) {\small{$w_1$}};
    \node[draw,circle,inner sep=.8pt,fill=black] at (w1) {}; 
        
    \coordinate (w3) at (4,2);
    \node[above, right] at (w3) {\small{$w_3$}};
    \node[draw,circle,inner sep=.8pt,fill=black] at (w3) {}; 
    
    \coordinate (w5) at (4,0);
    \node[below] at (w5) {\small{$w_5$}};
    \node[draw,circle,inner sep=.8pt,fill=black] at (w5) {}; 
    
    \coordinate (w6) at (5,1);
    \node[above,right] at (w6) {\small{$w_6$}};
    \node[draw,circle,inner sep=.8pt,fill=black] at (w6) {}; 
    
    \coordinate (w2) at (3.5,1.25);
    \node[left, below] at (w2) {\small{$w_2$}};
    
    \coordinate (w4) at (5,0);
    \node[below] at ($(w4)$) {\small{$w_4$}};
    \node[draw,circle,inner sep=.8pt,fill=black] at (w4) {}; 
    
    \coordinate (degg) at ($(4+2/3 , 2/3)$);
    \node[below] at ($(degg)+(-.05,-0.05)$) {\tiny{$d$}};
    \node[draw,circle,inner sep=.8pt,fill=black] at (degg) {};  
    
    \coordinate (u) at (5.3,.5);
    \node[left] at ($(u)$) {\tiny{$u$}};
    \node[draw,circle,inner sep=.8pt,fill=black] at (u) {};

  \draw[thin, black!70!, dashed] (u) --  ($5.5*(u)-4.5*(w1)$);
    \fill[black!90!] ($5.5*(u)-4.5*(w1)$) node [above]  {\small{$H_{1}$}}; 

 \draw[thin, black!70!, dashed] (u) --  ($4*(u)-3*(w4)$);
    \fill[black!90!] ($4*(u)-3*(w4)$) node [right]  {\small{$H_{4}$}}; 

\draw[thin, black!70!, dashed] (u) --  ($4.5*(u)-3.5*(w6)$);
    \fill[black!90!] ($4.5*(u)-3.5*(w6)$) node [left]  {\small{$H_{6}$}};  
   \coordinate (a) at (5,2);
    \node[] at ($(a)$) {\tiny{$M_a$}};     

\coordinate (f) at (3.65,-.65);
    \node[] at ($(f)$) {\tiny{$M_f$}};

\coordinate (sell) at
(intersection cs: first line={(w4) -- (u)},
                  second line={(w1) -- (w3)});
                  
\coordinate (6u) at
(intersection cs: first line={(w6) -- (u)},
                  second line={(w1) -- (w5)});

\draw[thick, loosely dashdotted]  (6u) -- (3,0);
   \draw[thick, loosely dotted]  (sell) -- (3.3,2.7);
    \draw[thin]  (w1) -- (degg);
    \draw[thin]  (w4) -- (w6);
    \draw [thin] (w3) -- (w4);
    \draw [thin] (w5) -- (w6);
    
    \draw[thin, dotted]  (degg) -- (3,1.5);
          
\end{tikzpicture}   
\end{center}
We remark that $n=8$ with $w_7,w_8$ as in Case~(2)
is not possible. 
Indeed, otherwise, we have  
$u \not\in \cone(w_2,\ldots,w_6,w_7,w_8)$,
where by Proposition~\ref{prop:ample}, the latter 
cone contains the moving cone $X$ and thus~$u$; 
a contradiction.
Moreover, we note that for any $w_\ell,w_{\ell+1}$ 
of Case~(2), we have
$$ 
\SAmple(X)
\ \subseteq \ 
Q(\gamma_{1,6,4})
\cap 
Q(\gamma_{1,\ell,4})
\cap 
Q(\gamma_{1,6,\ell+1}).
$$

\medskip
\noindent
\emph{Case (3)}.
Write $w_\ell = (x,y,z)$.
Then $\det(w_\ell,w_3,w_5)=1$ shows $x=1$.
Moreover, $\det(w_\ell,w_2,w_4)=1$ implies 
$y=d_1z$.
Finally,  $\det(w_1,w_3,w_{\ell+1})=1$
leads to $z=0$.
We arrive at $w_\ell = w_1$ and thus 
$w_\ell\notin M_c^\circ$,
a contradiction.

\medskip
\noindent
\emph{Case~(4)}. 
Write $w_\ell = (x,y,z)$.
Then $\det(w_\ell,w_3,w_5)=1$ and
$\det(w_1,w_6,w_{\ell+1})=1$
show $x=1$ and $z=0$.
Now $\det(w_\ell,w_6,w_4)=1$
yields~$d_1y=0$.
We distinguish the cases 
$d_1 = 0$ and $d_1 > 0$.

\medskip
\noindent
\emph{Case (4.1)}: We have $d_1 = 0$.
Here, we have the following situation, 
where, in the figure, $w_\ell = (1,y,0)$ 
lies on the dotted line
and $w_{\ell+1}$ on the dash-dotted line.
$$
[w_1,w_2,w_3,w_4,w_5,w_6,w_\ell,w_{\ell+1}]
\ = \
\left[
\begin{array}{cc|cc|cc|cc}
1 & -1    & 0 & 0   & 0 & 0    & 1  & -1  \\
0 & 1     & 1 & 0   & 0 & 1    & y  & 1-y   \\
0 & 1     & 0 & 1   & 1 & 0    & 0  & 1
\end{array}
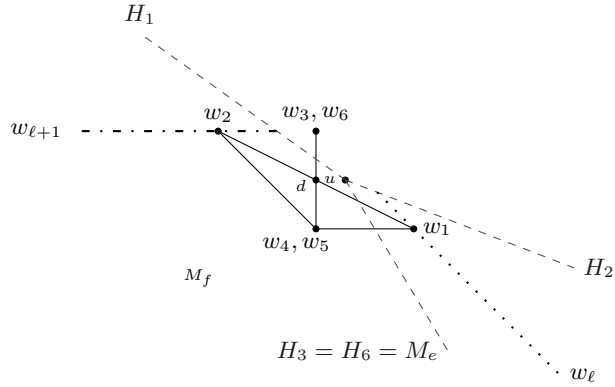
\right],
$$
\begin{center}
\begin{tikzpicture}[scale=1.3]

    \coordinate (w1) at (1,0);
    \node[right] at (w1) {\small{$w_1$}};
    \node[draw,circle,inner sep=.8pt,fill=black] at (w1) {}; 
    
    \coordinate (w6) at (0,1);
    \node[above] at (w6) {\small{$w_3, w_6$}};
    \node[draw,circle,inner sep=.8pt,fill=black] at (w6) {}; 
    
    \coordinate (w2) at (-1,1);
    \node[above] at (w2) {\small{$w_2$}};
    \node[draw,circle,inner sep=.8pt,fill=black] at (w2) {}; 
    
    \coordinate (w4) at (0, 0);
    \node[below] at ($(w4)+(-.2,0)$) {\small{$w_{4},w_5$}};
    \node[draw,circle,inner sep=.8pt,fill=black] at (w4) {};

    \coordinate (degg) at (0,.5);
    \node[left] at ($(degg)+(-.005,-0.06)$) {\tiny{$d$}};
    \node[draw,circle,inner sep=.8pt,fill=black] at (degg) {};

    \draw[thin]  (w1) -- (w2);
    \draw[thin]  (w1) -- (w4);
    \draw[thin]  (w4) -- (w6);
    \draw[thin]  (w2) -- (w4);

\coordinate (u) at (.3,.5);
    \node[left] at ($(u)$) {\tiny{$u$}};
    \node[draw,circle,inner sep=.8pt,fill=black] at (u) {}; 

    \draw[thin, black!70!, dashed] (u) --  ($4*(u)-3*(w1)$);
    \fill[black!90!] ($4*(u)-3*(w1)$) node [above]  {\small{$H_{1}$}}; 
    \draw[thin, black!70!, dashed] (u) --  ($2.8*(u)-1.8*(w2)$);
    \fill[black!90!] ($2.8*(u)-1.8*(w2)$) node [right]  {\small{$H_{2}$}}; 
  
    \draw[thin, black!70!, dashed] (u) --  ($4.5*(u)-3.5*(w6)$);
    \fill[black!90!] ($4.5*(u)-3.5*(w6)$) node [left]  {\small{$H_3 = H_6 = M_e$}};

    
    \coordinate (f) at (-1.2,-.5);
    \node[] at ($(f)$) {\tiny{$M_f$}};

\coordinate (2u) at
(intersection cs: first line={(w2) -- (u)},
                  second line={(w1) -- (w6)});
\draw[thick, loosely dotted]  (2u) -- (2.5,-1.5);

\coordinate (1u) at
(intersection cs: first line={(w1) -- (u)},
                  second line={(w2) -- (w6)});
\draw[thick, loosely dashdotted]  (1u) -- (-2.5,1);
    
\node[left] at (-2.5,1) {\small{$w_{\ell+1}$}};               
\node[right] at (2.5,-1.5){\small{$w_\ell$}};                    
            
\end{tikzpicture}  
\end{center}
Applying Proposition~\ref{prop:ample} 
to the resulting $X$-relevant faces 
of the present case, we arrive at
$$
\SAmple(X)
\ \subseteq \ 
Q(\gamma_{1,3,5})  
\cap 
Q(\gamma_{\ell,3,5})
 \cap 
Q(\gamma_{\ell,3,2}) 
\cap 
Q(\gamma_{\ell+1,3,1}).
$$

\medskip
\noindent
\emph{Case (4.2)}: We have $d_1 > 0$.
Then~$y=0$ must hold.  
This implies $w_\ell = w_1$ and $w_{\ell+1} = w_2$.
For the semiample cone, we have 
$$ 
\SAmple(X)
\ \subseteq \
Q(\gamma_{1,4,6}) \cap Q(\gamma_{1,2,6}).
$$
Moreover, the weights are arranged as in the 
figure below, where $w_2=w_{\ell+1}$ lies on the 
dotted line.
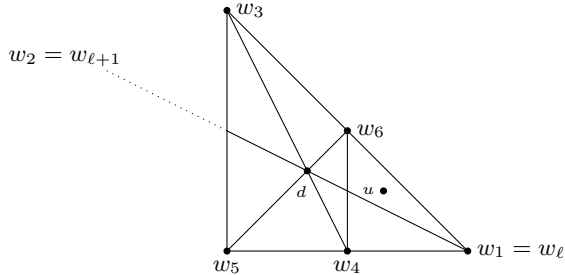
\begin{center}
\begin{tikzpicture}[scale=1.6]

    \coordinate (w1) at (6,0);
    \node[right] at (w1) {\small{$w_1=w_\ell$}};
    \node[draw,circle,inner sep=.8pt,fill=black] at (w1) {}; 
        
    \coordinate (w3) at (4,2);
    \node[above, right] at (w3) {\small{$w_3$}};
    \node[draw,circle,inner sep=.8pt,fill=black] at (w3) {}; 
    
    \coordinate (w5) at (4,0);
    \node[below] at (w5) {\small{$w_5$}};
    \node[draw,circle,inner sep=.8pt,fill=black] at (w5) {}; 
    
    \coordinate (w6) at (5,1);
    \node[above,right] at (w6) {\small{$w_6$}};
    \node[draw,circle,inner sep=.8pt,fill=black] at (w6) {}; 
    
    \coordinate (w2) at (3.5,1.25);
    \node[left] at (3.2,1.6) {\small{$w_2=w_{\ell+1}$}};
    
    \coordinate (w4) at (5,0);
    \node[below] at ($(w4)$) {\small{$w_4$}};
    \node[draw,circle,inner sep=.8pt,fill=black] at (w4) {}; 
    
    \coordinate (degg) at ($(4+2/3 , 2/3)$);
    \node[below] at ($(degg)+(-.05,-0.05)$) {\tiny{$d$}};
    \node[draw,circle,inner sep=.8pt,fill=black] at (degg) {};  
    
    \coordinate (u) at (5.3,.5);
    \node[left] at ($(u)$) {\tiny{$u$}};
    \node[draw,circle,inner sep=.8pt,fill=black] at (u) {}; 


    \draw[thin]  (w1) -- (4,1);
    \draw[thin]  (w1) -- (w5);
    \draw[thin]  (w4) -- (w6);
    \draw[thin]  (w3) -- (w5);
    \draw [thin] (w3) -- (w4);
    \draw [thin] (w5) -- (w6);
    \draw [thin] (w1) -- (w3);
    
    \draw[thin, dotted]  (4,1) -- (3,1.5);
          
\end{tikzpicture}   
\end{center}

\medskip

Subsuming the discussion so far, 
we see that only the Cases~(2) and~(4) 
allow weights $w_i$, where $i \ge 7$.
The remaining task is to check 
in which ways these cases can be combined.
So, let us go through the possible 
constellations of the pairs $w_i,w_{i+1}$ 
for $i = 7, 9, \ldots, n-1$.

\medskip
\noindent
\emph{All pairs $w_i,w_{i+1}$, where $i = 7, 9, \ldots, n-1$,
are from Case~(2)}.
In the discussion of Case~(2) we have seen that 
$n \ge 10$ must hold.
For odd $i \ge 7$, we have 
$w_i = (x_i,1,0)$ and $w_{i+1}=(d_1-x_i,0,1)$,
where we may assume
$x_7 \ge x_9 \ge \ldots \ge x_{n-1}$.
Now, the ample class $u$ lies in the moving 
cone.
Proposition~\ref{prop:ample} yields 
$$ 
u \in \cone(w_2,w_3, \ldots,w_{n-1},w_n)^\circ.
$$
We conclude $w_7 \in Q(\gamma_{1,6})^\circ$ 
and $w_{n} \in Q(\gamma_{1,4})^\circ$.
This in turn implies $x_7 > d_1$ and 
$x_{n-1} < 0$.
Moreover, $Q(\gamma_{7,n})$ is a bounding face 
of the moving cone.
Thus, we obtain
$$ 
u 
\ \in \ 
Q(\gamma_{1,7,4})^\circ \cap Q(\gamma_{1,6,n})^\circ
\setminus
Q(\gamma_{1,7,n})
\ \subseteq \ 
Q(\gamma_{7,2,n})^\circ.
$$
We conlude that $\gamma_{7,2,n}$ is  $X$-relevant.
Applying Proposition~\ref{prop:smooth}~(iii)
and the estimates for $x_7$ and $x_n$ just 
obtained, we arrive at a contradiction,
showing that the present setting can't occur:
$$ 
1 
\ = \ 
\det(w_7,w_2,w_n) 
\ = \ 
d_1 - x_{n-1} + x_7 - d_1 +1
\ \ge \ 2.
$$

\medskip
\noindent
\emph{There is an even $7 < k \le n$ such that 
for $i = 7, \ldots, k-1$, the pairs 
$w_i,w_{i+1}$ are from Case~(4.1) and for all odd
$j = k+1, \ldots n-1$, we have $w_j=w_3$ and $w_{j+1} = w_4$}.
Then, for the odd $i = 7, \ldots, k-1$,
we have $w_i = (1,y_i,0)$ and $w_{i+1} = (-1,1-y_{i},1)$,
where we may assume $y_7 \ge \ldots \ge y_{k-1}$.
Set
$$
\alpha \sei \max(0, y_7), 
\quad
\tilde{w}_1 \sei (1,\alpha,0),
\qquad
\beta\sei\min(0, y_{k-1}),
\quad
\tilde{w}_{k} \sei (-1,1-\beta,1).
$$ 
Then $\tilde{w}_1$ and $\tilde{w}_{k}$ are 
degrees of variables and they are closest
to $w_3$ in the sense that 
$\tilde{w}_1 \in \cone(w_3,w_i)$ 
and
$\tilde{w}_k \in \cone(w_3,w_{i+1})$ 
holds for $i = 1$ and $i = 7, \ldots, k-1$;
see the figure in Case~(4.1).
Recall that the semiample cone is 
contained in the intersection of
$Q(\gamma_{1,3,5})$ and $\cone(d,w_1,w_3)$.
We even claim 
$$ 
\SAmple(X) 
\ = \ 
\cone(\tilde{w}_1,w_3,w_5) \cap \cone(\tilde{w}_1,w_3,\tilde{w}_{k}).
$$
By the definition of $\tilde w_1$ and $\tilde w_k$, 
we only have to show that both cones are images 
of $X$-relevant faces.
For the first one this is clear.
We discuss the second one.
Observe that we have 
$$ 
u 
\ \in \ 
\cone(\tilde{w}_1,w_3,\tilde{w}_{k})
\cup 
\cone(\tilde{w}_1,w_4,\tilde{w}_{k})
$$ 
Thus, according to Proposition~\ref{prop:Qfact},
the task is to show that 
$\cone(\tilde{w}_1,w_4,\tilde{w}_{k})$
is not the image of an $X$-relevant face.
Indeed, this would contradict 
Lemma~\ref{lem:nosmootharea1}
applied to an $X$-relevant 
face projecting onto 
$$
\cone(\tilde{w}_2,w_3,\tilde{w}_{k-1}),
\qquad
\tilde{w}_2 \sei (-1,1-\alpha,1),
\quad
\tilde{w}_{k-1} \sei (1,\beta,0).
$$
Now, the coordinate change on $K = \ZZ^3$
given by the following unimodular matrix
and suitably renumbering of variables
leads to the setting of 
Theorem~\ref{thm:fullpic3}:
$$
\left[
\begin{array}{ccc}
1 & 0 & 1  \\
-\beta & 1  & \alpha-\beta-1   \\
0 & 0  & 1
\end{array}
\right]. 
$$

\medskip
\noindent
\emph{There is an even $7 < k < n$ such that 
for $i = 7, \ldots, k-1$, the pairs 
$w_i,w_{i+1}$ are from Case~(4.1) and for all odd
$j = k+1, \ldots ,n-1$ the pair $w_j,w_{j+1}$ is 
from Case~(2)}.
Note that we have $d_1 = 0$ and the weights are of the form
$$
w_i = (1,y_i,0),
\quad
w_{i+1} = (-1,1-y_i,1),
\qquad
w_j = (x_j,1,0),
\quad
w_{j+1} = (-x_j,0,1),
$$
where we may assume $y_7 \ge \ldots \ge y_{k-1}$ 
and $x_{k+1} \ge \ldots \ge x_{n-1}$.
The weights are arranged as follows,
where the $w_i$ for $i = 7, \ldots, k-1$ 
and the $w_j$ for $j = k+1, \ldots, n-1$
lie on the dotted line:
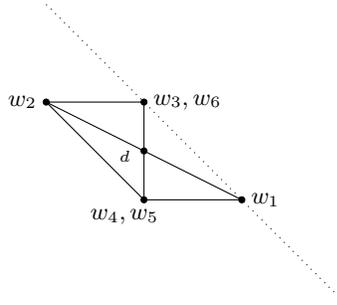
\begin{center}
\begin{tikzpicture}[scale=1.3]

    \coordinate (w1) at (1,0);
    \node[right] at (w1) {\small{$w_1$}};
    \node[draw,circle,inner sep=.8pt,fill=black] at (w1) {}; 
    
    \coordinate (w6) at (0,1);
    \node[above,right] at (w6) {\small{$w_3, w_6$}};
    \node[draw,circle,inner sep=.8pt,fill=black] at (w6) {}; 
    
    \coordinate (w2) at (-1,1);
    \node[left] at (w2) {\small{$w_2$}};
    \node[draw,circle,inner sep=.8pt,fill=black] at (w2) {}; 
    
    \coordinate (w4) at (0, 0);
    \node[below] at ($(w4)+(-.2,0)$) {\small{$w_{4},w_5$}};
    \node[draw,circle,inner sep=.8pt,fill=black] at (w4) {};
        

    \coordinate (degg) at (0,.5);
    \node[left] at ($(degg)-(.05,0.05)$) {\tiny{$d$}};
    \node[draw,circle,inner sep=.8pt,fill=black] at (degg) {};  
  
   
    \draw[thin]  (w1) -- (w2);
    \draw[thin]  (w1) -- (w4);
    \draw[thin]  (w4) -- (w6);
    \draw[thin]  (w2) -- (w4);
    \draw[thin]  (w2) -- (w6);
    \draw[thin, dotted] (-1,2) -- (2,-1);
        
\end{tikzpicture}  
\end{center}
The discussion on the Cases~(4.1) and~(2) 
performed so far shows that for 
all odd $i = 7, \ldots, k-1$ and 
$j = k+1, \ldots, n-1$, the semiample cone 
of $X$ satisfies
$$
\SAmple(X) 
\ \subseteq \ 
Q(\gamma_{3,5,i})
\cap 
Q(\gamma_{1,4,j})
\ \subseteq \ 
Q(\gamma_{5,i,j}).
$$
Since the semiample cone is full-dimensional, 
we see that~$\gamma_{5,i,j}$ is $X$-relevant.
Thus, we may apply Proposition~\ref{prop:smooth}
and obtain~
$$
1
\ = \ 
\det(w_5,w_i,w_j)
\ = \ 
1-y_ix_j.
$$
This leaves us with $y_i=0$ for $i = 7, \ldots, k-1$
or $x_j = 0$ for $j = k+1, \ldots, n-1$.
If all the $x_j$ vanish, then we are in the case just 
treated.
So, assume that all the $y_i$ vanish.
Then we have $w_i = w_1$ and $w_{i+1} = w_2$ for 
$i = 7, \ldots, k-1$.
Set 
$$
\alpha \sei \max(0,x_{k+1}),
\quad 
\tilde w_{3}  \sei (\alpha,1,0),
\qquad
\beta \sei \min(0,x_{n-1}),
\quad 
\tilde w_{n}  \sei (-\beta,0,1).
$$ 
Then $\tilde w_3$ and $\tilde w_n$ are 
the degrees of the variables sitting closest 
to $w_1$ and among the $w_j$ and $w_{j+1}$ 
with $j=3$ or $j = k+1,\ldots, n-1$.
We claim that the semiample cone 
is given by
$$ 
\SAmple(X)
\ = \ 
\cone(w_1,\tilde w_3,w_2)
\cap
\cone(w_1,\tilde w_3, \tilde w_n).
$$
As in the preceding case,
we only have to show that both cones 
are projected $X$-relevant faces.
For the first one this is clear.
We turn to the second one.
For sure we have
$$ 
u 
\ \in \ 
\cone(w_1,\tilde{w}_3,\tilde{w}_{n})
\cup 
\cone(\tilde{w}_3,w_2,\tilde{w}_{n}) .
$$ 
We verify
$u \not\in \cone(\tilde{w}_3,w_2,\tilde{w}_{n})$.
Otherwise, because of
$\det(\tilde{w}_3,w_2,\tilde{w}_{n}) = 1 + \alpha - \beta$,
Proposition~\ref{prop:smooth}~(iii) 
yields $\alpha = \beta = 0$; a contradiction.
Thus, $\cone(w_1,\tilde w_3, \tilde w_n)$ 
is the image of an $X$-relevant face. 
%
Now, the coordinate change on $K = \ZZ^3$ given 
by the following unimodular matrix 
and suitably renumbering of variables
leads to the setting of Theorem~\ref{thm:fullpic3}:
$$
\left[
\begin{array}{ccc}
0 & 1 & 0  \\
1 & -\beta & \alpha    \\
0 & 0  & 1
\end{array}
\right]. 
$$

\medskip
\noindent
\emph{There is an even $7 < k < n$ such that 
for $i = 7, \ldots, k-1$, the pairs 
$w_i,w_{i+1}$ are from Case~(4.2) and for all odd
$j = k+1, \ldots n-1$ the pair $w_j,w_{j+1}$ is 
from Case~(2)}.
This means 
$$
w_i = (1,0,0),
\quad
w_{i+1} = (d_1-1,1,1),
\qquad
w_j = (x_j,1,0),
\quad
w_{j+1} = (d_1-x_j,0,1),
$$
A coordinate change on $K = \ZZ^3$ given by
the following unimodular matrix 
and suitably renumbering of variables
leads to the preceding case:
$$
\left[
\begin{array}{ccc}
0 & 1 &-1  \\
1 & 0 & 1-d_1 \\
0 & 0 & 1  
\end{array}
\right].
$$
\end{proof}

\begin{proof}[Proof of Theorem~\ref{thm:fullpic3}, Part~II]
Let~$X$ arise from Construction~\ref{constr:intquad} 
with the input data specified in Theorem~\ref{thm:fullpic3}.
Consider the toric embedding $X \subseteq Z$ 
provided by Construction~\ref{constr:intquad}.
From Remark~\ref{rem:Xfaces}, we infer
$\tau_X = \tau_Z$ for the semiample cones. 
Thus, for the divisor class $w = (1,a+1,1)$,
a representing toric divisor $E$ on $Z$ 
and its restriction $D$ on $X$, we obtain 
a commutative diagram 
$$ 
\xymatrix{
X
\ar@{}[r]|\subseteq
\ar[d]_{\varphi_D}
&
Z
\ar[d]^{\varphi_E}
\\
X(D)
\ar@{}[r]|\subseteq
&
Z(E)
\ar@{}[r]|{\hspace*{-2cm}\cong}
& 
{\PP(\mathcal{O}_{\PP_{l-1}}(b_1) \oplus \ldots \oplus \mathcal{O}_{\PP_{l-1}}(b_k)),}
}
$$
where the inclusions are closed embeddings,
$l$ is the number of coordinates of 
$\overline{Z} = \KK^{n}$ of degree
$w_1=(0,1,0)$, call them $f_1, \ldots, f_l$,
and $k$ is the number of coordinates
whose degree is located on the line segment
$\cone(w_5,w_8)$, 
call them $h_1, \ldots, h_k$.
Then we have $n = 2l+2k$
and $\deg(h_i) = (1,b_i,0)$
with $0=b_1\le b_2\le\ldots\le b_k=a$.
Using local trivializations, 
we see that $X$ projects onto the base
$Z(E)$, which means $X(D) = Z(E)$.
Moreover, on each fiber $\varphi_E^{-1}([z])$, 
the relation~$g$ becomes a linear form 
in the coordinates~$T_i$ different from $f_i$ and $h_j$
and thus cuts out a hyperplane of 
$\varphi_E^{-1}([z]) \cong \PP_{l+k-1}$.
Consequently $\varphi_D \colon X \to X(D)$ 
is as claimed.
\end{proof}

\begin{corollary}
\label{cor:fujitapic3full}
Let~$X$ be a Fano smooth full intrinsic quadric
of Picard number three.
Then every numerically effective divisor 
on~$X$ is base point free.
In particular,~$X$ fulfills Fujita's freeness 
conjecture.
\end{corollary}

\begin{proof}
We may assume that $X$ arises 
from Construction~\ref{constr:intquad}
with the input data specified in 
Theorem~\ref{thm:fullpic3}.
As in the proof of Corollaries~\ref{cor:bpfsat}
and~\ref{cor:bpfquadricssatrhoX2}, 
we consider the monoid of base 
point free divisor classes and its 
combinatorial description:
$$ 
\BPF(X) 
\ = \ 
\bigcap_{\gamma_0 \in \cov(X)} Q(\gamma_0 \cap \ZZ^{n})
$$
Again the task is to show that for all~$\gamma_0\in\cov(X)$,
the monoid~$Q(\gamma_0\cap E)\subseteq Q(\lin(\gamma_0)\cap E)$
is saturated.
For three-dimensional~$\gamma_0$, 
this is due to Proposition~\ref{prop:smooth}~(iii).
If $\gamma_0\in\cov(X)$
is not three-dimensional,
then we have~$\gamma_0 = \gamma_{i,i+1,j,j+1}$ 
with
$w_i=(0,1,0)$, $w_{i+1}=(1,a-1,1)$ 
and
$w_j=(0,b,1)$, $w_{j+1}=(1,a-b,0)$
for some $0\le b\le a$.
One directly checks that the corresponding monoid
$Q(\gamma_0\cap E)\subseteq Q(\lin(\gamma_0)\cap E)$
is saturated.
\end{proof}



\begin{bibdiv}
\begin{biblist}

\bib{AlIl}{article}{
    author = {Altmann, Klaus},
    author =  {Ilten, Nathan},
     TITLE = {Fujita's freeness conjecture for $T$-varieties of complexity one},
      eprint = {arXiv:1712.09927},
}

\bib{ACO}{article}{
    AUTHOR = {Andreatta, Marco},
    author =  {Chierici, Elena},
    author = {Occhetta, Gianluca},
     TITLE = {Generalized {M}ukai conjecture for special {F}ano varieties},
   JOURNAL = {Cent. Eur. J. Math.},
    VOLUME = {2},
      YEAR = {2004},
    NUMBER = {2},
     PAGES = {272--293},
}
%
\bib{ArDeHaLa}{book}{
   author={Arzhantsev, Ivan},
   author={Derenthal, Ulrich},
   author={Hausen, J\"urgen},
   author={Laface, Antonio},
   title={Cox rings},
   series={Cambridge Studies in Advanced Mathematics},
   volume={144},
   publisher={Cambridge University Press, Cambridge},
   date={2015},
   pages={viii+530},
   isbn={978-1-107-02462-5},
}

\bib{Ba}{article}{
   author={Batyrev, Victor V.},
   title={On the classification of smooth projective toric varieties},
   journal={Tohoku Math. J. (2)},
   volume={43},
   date={1991},
   number={4},
   pages={569--585},
}






\bib{BeHa:2007}{article}{
   author={Berchtold, Florian},
   author={Hausen, J\"urgen},
   title={Cox rings and combinatorics},
   journal={Trans. Amer. Math. Soc.},
   volume={359},
   date={2007},
   number={3},
   pages={1205--1252},
}


\bib{bocadedr}{article}{
   author={Bonavero, Laurent},
   author={Casagrande, Cinzia},
   author={Debarre, Olivier},
   author={Druel, Stephane},
   title={Sur une conjecture de Mukai},
   journal={Comment.~Math.~Helv.},
   volume={78},
   date={2003},
   pages={601--626},
}



\bib{Bo2}{article}{
   author={Bourqui, David},
   title={La conjecture de Manin g\'eom\'etrique pour une famille de quadriques
   intrins\`eques},
   language={French, with English and French summaries},
   journal={Manuscripta Math.},
   volume={135},
   date={2011},
   number={1-2},
   pages={1--41},
}



\bib{Cas}{article}{
    AUTHOR = {Casagrande, Cinzia},
     TITLE = {The number of vertices of a {F}ano polytope},
   JOURNAL = {Ann. Inst. Fourier (Grenoble)},
    VOLUME = {56},
      YEAR = {2006},
    NUMBER = {1},
     PAGES = {121--130},
}


\bib{Cox}{article}{
   author={Cox, David A.},
   title={The homogeneous coordinate ring of a toric variety},
   journal={J. Algebraic Geom.},
   volume={4},
   date={1995},
   number={1},
   pages={17--50},
}


\bib{Ca}{article}{
   author = {Casagrande, Cinzia},
   title = {On the birational geometry of Fano 4-folds},
   journal = {Math. Ann.},
   date={2013},
   number={355},
   pages={585--628},
}


\bib{EL1}{article}{
    AUTHOR = {Ein, Lawrence},
    author = {Lazarsfeld, Robert},
     TITLE = {Global generation of pluricanonical and adjoint linear series
              on smooth projective threefolds},
   JOURNAL = {J. Amer. Math. Soc.},
    VOLUME = {6},
      YEAR = {1993},
    NUMBER = {4},
     PAGES = {875--903},
}
		


\bib{Fa}{article}{
   author = {Fahrner, Anne},
   title = {Smooth Mori dream spaces of small Picard number},
   journal = {Doctoral Dissertation, Universit\"at T\"ubingen},
   date={2017},
   eprint = {https://publikationen.uni-tuebingen.de}
}



\bib{FaHaNi}{article}{
   author = {Fahrner, Anne},
   author = {Hausen, J\"urgen},
   author = {Nicolussi, Michele},
   title = {Smooth projective varieties with a torus action of complexity 1 and Picard number 2},
   journal = {to appear in  Ann. Sc. Norm. Super. Pisa Cl. Sci.},
   eprint = {arXiv:1602.04360},
}


\bib{fujino}{article}{
    AUTHOR = {Fujino, Osamu},
     TITLE = {Notes on toric varieties from {M}ori theoretic viewpoint},
 journal = {Tohoku Math.~J.~(2)},
    VOLUME = {55},
    number={4},
     PAGES = {551--564},
      YEAR = {2003},
}


\bib{Fu}{article}{
    AUTHOR = {Fujita, Takao},
     TITLE = {On polarized manifolds whose adjoint bundles are not
              semipositive},
 BOOKTITLE = {Algebraic geometry, {S}endai, 1985},
    SERIES = {Adv. Stud. Pure Math.},
    VOLUME = {10},
     PAGES = {167--178},
 PUBLISHER = {North-Holland, Amsterdam},
      YEAR = {1987},
}



		
\bib{HaHe}{article}{
   author={Hausen, J\"urgen},
   author={Herppich, Elaine},
   title={Factorially graded rings of complexity one},
   conference={
      title={Torsors, \'etale homotopy and applications to rational points}
   },
   book={
      series={London Math. Soc. Lecture Note Ser.},
      volume={405},
      publisher={Cambridge Univ. Press, Cambridge},
   },
   date={2013},
   pages={414--428},
}



\bib{HuKe}{article}{
author = {Hu, Yi},
author = {Keel, Sean},
journal = {Michigan Math. J.},
number = {1},
pages = {331--348},
title = {Mori dream spaces and GIT},
volume = {48},
year = {2000},
}


\bib{Kaw}{article}{
    AUTHOR = {Kawamata, Yujiro},
     TITLE = {On {F}ujita's freeness conjecture for {$3$}-folds and
              {$4$}-folds},
   JOURNAL = {Math. Ann.},
    VOLUME = {308},
      YEAR = {1997},
    NUMBER = {3},
     PAGES = {491--505},
}


\bib{Kl}{article}{
   author={Kleinschmidt, Peter},
   title={A classification of toric varieties with few generators},
   journal={Aequationes Math.},
   volume={35},
   date={1988},
   number={2--3},
   pages={254--266},
}


\bib{MoMu}{article}{
author = {Mori, Shigefumi},
author = {Mukai, Shigeru},
journal = {Manuscripta mathematica},
pages = {147-162},
title = {Classification of Fano 3-Folds with $B_2 \ge 2$},
volume = {36},
year = {1981},
}



\bib{Mukai}{article}{
   author={Mukai, S.},
   title={Problems on characterization of the complex projective space},
   journal={In: Birational Geometry of Algebraic
Varieties, Open Problems, Katata, August 22--27},
   date={1988},
   pages={57--60},
}



\bib{Re}{article}{
    AUTHOR = {Reider, Igor},
     TITLE = {Vector bundles of rank {$2$} and linear systems on algebraic
              surfaces},
   JOURNAL = {Ann. of Math. (2)},
    VOLUME = {127},
      YEAR = {1988},
    NUMBER = {2},
     PAGES = {309--316},
}


\bib{Wis}{article}{
   author={Wi\'sniewski, Jaros\l aw A.},
   title={On a conjecture of Mukai},
   journal={Manuscripta Math.},
   volume={68},
   date={1990},
   number={2},
   pages={135--141},
   issn={0025-2611},
   review={\MR{1063222}},
}


\bib{yezhu2}{article}{
   author = {Ye, Fei},
   author = {Zhu, Zhixian},
    title = {On Fujita's freeness conjecture in dimension 5},
   eprint = {arXiv:1511.09154},
     year = {2015},
}

\end{biblist}
\end{bibdiv}
\end{document}